\documentclass{amsart}
\newif\ifdraft

\draftfalse

\usepackage[T1]{fontenc}
\usepackage[hidelinks]{hyperref}
\usepackage{amssymb}
\usepackage{eurosym}
\usepackage{amsfonts}
\usepackage{amsmath}
\usepackage{geometry}
\usepackage[usenames,dvipsnames]{xcolor}
\usepackage{enumerate}
\usepackage{todonotes}

\usepackage{mathrsfs} % for mathscr ..

\newtheorem{theorem}{Theorem}
\theoremstyle{plain}

\newtheorem{corollary}[theorem]{Corollary}

\newtheorem{definition}[theorem]{Definition}

\newtheorem{lemma}[theorem]{Lemma}

\newtheorem{remark}[theorem]{Remark}

\numberwithin{equation}{section}

\newcommand{\E}{\mathbb{E}}

\def\D{\mathscr{D}} %    write \D^\{2\alpah}_X$ for controlled path spaces
\def\Cr{\mathscr{C}} %   write $\Cr([0,T],V)$ for rough spaces over some Banach space V
\def\R{\mathbb{R}}   %   reals

\def\drp{e} % dimension of the rough path
\def\dode{d} % dimension of the state space
\def\dbm{d_B} % dimension of the Feynman-Kac Brownian motion
\newcommand{\Lip}{C_b}

\begin{document}
\def\cprime{$'$}
\def\cprime{$'$}

\title[SPDEs and rough paths]{Stochastic partial differential equations: a rough path view}

\author{Joscha Diehl, Peter K. Friz and Wilhelm Stannat}
\address{JD and WS are affiliated to TU Berlin. PKF is corresponding author (friz@math.tu-berlin.de) and affiliated to TU and WIAS Berlin.}

\begin{abstract}
We discuss regular and weak solutions to rough partial differential equations (RPDEs), thereby providing a (rough path-)wise view 
on important classes of SPDEs. In contrast to many previous works on RPDEs, our definition gives honest meaning to RPDEs as integral equation, 
based on which we are able to obtain existence, uniqueness and stability results. The case of weak ``rough" forward equations,
may be seen as robustification of the (measure-valued) Zakai equation in the rough path sense. Feynman-Kac representation for RPDEs, in formal analogy to similar classical results in SPDE theory, play an important role. 
\end{abstract}

\keywords{stochastic partial differential equations, Zakai equation, Feynman-Kac formula, rough partial differential equations, rough paths}
\subjclass{60H15}
\maketitle

%\tableofcontents

\section{Introduction}

Consider a diffusion process $X$ on $\mathbb{R}^{d}$ with
generator given by a second order differential operator $L$. In its simplest
form, the Feynman-Kac formula asserts that, for suitable data $g$,%
\begin{equation}
u\left( t,x\right) = \mathbb{E}^{t,x}\left[ g\left( X_{T}\right) \right]
,\,\,\,\,\,t\leq T,x\in \mathbb{R}^{d},  \label{FKintro}
\end{equation}%
solves a parabolic partial differential equation, namely the
terminal value problem 
\begin{equation*}
\begin{cases}
-\partial_{t}u_{t} & =Lu_{t} \\ 
u \left( T,\cdot \right) & =g.%
\end{cases}
\end{equation*}%
(Below we will consider slightly more general operators including zero order term, causing additional exponential factors in the Feynman-Kac formula.) On the other hand, the law of $X_{t}$ started at $X_{0}=x$, solves the forward (or Fokker-Planck) equation 
\begin{equation*}
\begin{cases}
\partial_{t}\rho _{t} & =L^{\ast }\rho _{t} \\ 
\rho \left( 0,\cdot \right) & =\delta _{x}.%
\end{cases}%
\end{equation*}%
Formally at least, an infinitesimal version of (\ref{FKintro}) is given by
\begin{equation*}
\partial_{t}\, \left\langle u_{t},\rho _{t}\right\rangle =\left\langle
-Lu_{t},\rho_{t}\right\rangle +\left\langle \,u_{t},L^{\ast }\rho
_{t}\right\rangle =0,
\end{equation*}%
and indeed the resulting duality  $\left\langle u_{T},\rho _{T}\right\rangle
=\left\langle u_{0,}\rho _{0}\right\rangle$ is nothing than restatement of (\ref{FKintro}), at $t=0$.

\bigskip

In both cases, forward and backward, there may not exist a classical $%
C^{1,2}$ solution. Indeed, it suffices to consider the case of degenerate $%
X$ so that $\rho_t$ remains a measure; in the backward case consider $%
g\notin C^{2}$. In both cases one then needs a concept of weak solutions. 
A natural way to do this, consists in testing the equation in space; that is, to consider the evolution for $\left\langle u_{t},\phi \right\rangle $ and $\left\langle\rho _{t},f\right\rangle $ where $\phi $ and~$f$ are suitable test functions defined on $\mathbb{R}^{d}$.

\bigskip

Applications from filtering theory lead to (backward) SPDEs of the form 
\begin{equation*}
\begin{cases}
-du_{t} & =L[u_{t}]dt+\Gamma \lbrack u_{t}]\circ dW_{t} \\ 
u\left( T,\cdot \right) & =g,%
\end{cases}%
\end{equation*}%
where $W=(W^{1},\dots ,W^{e})$ and $\Gamma =\left( \Gamma _{1},\dots ,\Gamma
_{e}\right) $ are first order differential operators,\footnote{Write $\Gamma
\lbrack u]\circ dW=\sum_{k=1}^{e}\Gamma _{k}[u]\circ dW^{k}$.} in duality
with the forward (or Zakai) equation%
\begin{equation*}
\begin{cases}
d\rho _{t} & =L^{\ast }[\rho _{t}]dt+\Gamma ^{\ast }[\rho _{t}]\circ dW_{t}
\\ 
\rho \left( 0,\cdot \right) & =\delta _{x}.%
\end{cases}%
\end{equation*}%
Such SPDEs were studied extensively in classical works
\cite{Kunita1982,bib:rozovksii,bib:pardoux}.
It is a natural question, studied for
instance in a series of papers by Gy\"{o}ngy \cite{Gy88,Gy89}, to what extent such SPDEs are
approximated by (random) PDEs, upon replacing the (Stratonovich)
differential $dW=dW\left( \omega \right) $ by $\dot{W}^{\varepsilon }\left(
\omega \right) dt$, given a suitable family of smooth approximation $\left(
W^{\varepsilon }\right) $ to Brownian motion. In recent works \cite{FO14, DOR}, also \cite[Ch.12]{FH14}, it was shown
that the backward solutions $u^{\varepsilon }$, interpreted as viscosity
solution (assuming $g\in C_{b}$) actually converge locally uniformly, with
limit $u$ only depending on the rough path limit of $\left( W^{\varepsilon
}\right) $. Writing $\mathbf{W=(}W\mathbf{,\mathbb{W})}$ for such a (deterministic!) rough
path (see e.g. \cite{FH14} for notation) say, $\alpha $-H\"{o}lder, for $%
1/3\,<\alpha <1/2$) the question arises if one can give an honest meaning to
the equations %
\begin{eqnarray}
-du_{t} &=&L[u_{t}]dt+\Gamma \lbrack u_{t}]d\mathbf{W}_{t}\mathbf{,}
\label{u_intro} \\
d\rho _{t} &=&L^{\ast }[\rho _{t}]dt+\Gamma ^{\ast }[\rho _{t}]d\mathbf{W}%
_{t}\mathbf{.}  \notag
\end{eqnarray}%
In the aforementioned works, these ``rough partial differential equations" (RPDEs) had only formal meaning. The actual definition was then 
given either in terms of a (flow)transformed equation in the spirit of Kunita (e.g.  \cite{FO14}, also \cite[p.177]{FH14}) or
in terms of a unique continuous extension of the PDE solution as function of driving noise, \cite{CFO, FO14}. 
\bigskip

There are two difficulties with such rough partial differential equations. The first one is the temporal roughness of $%
\mathbf{W}$, a problem that has been well-understood from the rough path
analysis of SDEs. Indeed, following Davie's approach to RDEs \cite{bib:davie}, the
(rough) pathwise meaning of%
\begin{equation*}
dX=\beta \left( X\right) d\mathbf{W}
\end{equation*}%
is, by definition, and writing $X_{s,t}=X_{t}-X_{s}$ for path increments, 
\begin{equation*}
X_{s,t}=\beta \left( X_{s}\right) W_{s,t}+\beta ^{\prime }\beta \left(
X_{s}\right) \mathbb{W}_{s,t}+o\left( \left\vert t-s\right\vert \right) .
\end{equation*}%
Under suitable assumptions on $\beta $, uniqueness, local/global existence
results are well-known. This quantifies that
statement that $X$ is controlled by $W$, with ``Gubinelli derivative" $\beta
\left( X\right) $, and in turn implies the integral representation in terms of a bona-fide rough integral (cf. \cite[Ch.4]{FH14})%
\begin{equation*}
X_{t}-X_{s}=\int_{s}^{t}\beta \left( X\right) d\mathbf{W}=\lim \sum_{\left[
u,v\right] \in P}\beta \left( X_{u}\right) W_{u,v}+\beta ^{\prime }\beta
\left( X_{u}\right) \mathbb{W}_{u,v}\text{.}
\end{equation*}%
This suggests that the meaning of the backward equation (\ref{u_intro}) is 
\begin{equation*}
u\left( s,x\right) -u\left( t,x\right)
=\int_{s}^{t}L[u_{r}]dr+\int_{s}^{t}\Gamma \lbrack u_{r}]d\mathbf{W}_{r},
\end{equation*}%
provided $u$ is sufficiently regular (in space) such as to make $L[u],\Gamma
\lbrack u]$ meaningful, {\it and} provided the last term makes sense as rough integral. The other difficulty is
exactly that $u$ may not be regular in space so that $L[u],\Gamma \lbrack u]$ require a weak meaning. More precisely, we propose the following spatially
weak\footnote{%
There is no probability here, for $\mathbf{W}$ is a \textit{deterministic}
rough path. Nevertheless, with a view to later applications to SPDEs and to
avoid misunderstandings, let us emphasize that in this paper ``weak" is always understood as
``analytically weak".} formulation, of the form 
\begin{equation*}
\left\langle u_{s},\phi \right\rangle -\left\langle u_{t},\phi \right\rangle
=\int_{s}^{t}\left\langle u_{r},L^{\ast }\phi \right\rangle
dr+\int_{s}^{t}\left\langle u_{r},\Gamma ^{\ast }\phi \right\rangle d\mathbf{%
W}_{r},
\end{equation*}%
where, again, we can hope to understand the last term as rough integral.
(Everything said for backward equations translates, mutatis mutandis, to the
forward setting.)

\bigskip 
The main result of this paper is that - in all cases - one has existence and
uniqueness results. Loosely speaking (and subject to suitable regularity
assumptions on the coefficients of $L,\Gamma $; but no ellipticity assumptions) we have

{\color{red} }

\begin{theorem}
\label{thmmainloose} For nice terminal data $g$ there exists a unique (spatially) regular solution
to the backward RPDE. Similarly, for nice initial data $\rho_0$ (with nice density $p_0$, say) the forward RPDE has a unique (spatially) regular solution. \\
If the terminal data $g$ is only bounded and continuous, we have existence and uniqueness of a weak solution to the backward RPDE.
Similarly, if the initial data  $\rho_0$ of the forward RPDE is only a finite measure, we have existence and uniqueness of a weak (here: measure-valued) 
solution to the forward RPDE. \\
In all cases, the (unique) solution depends continuously on the driving rough path and we have Feynman--Kac type representation formulae.
\end{theorem}

Let us briefly discuss the strategy of proof. In all cases (regular/weak,
forward/backward) existence of a solution is verified via an explicit
Feynman--Kac type formula, based on a notion of ``hybrid" It\^o/rough
differential equation, which already appeared in previous works
\cite{CDFO13,DOR}, see also \cite{FH14}. We then use regular forward existence to show weak backward uniqueness (Theorem \ref{thm:roughBackwardEquation}),
which actually requires us to work with exponentially decaying test functions. %
Next, regular backward existence leads to weak (actually, measure-valued) forward
uniqueness (Theorem \ref{thm:roughForwardEquation}), here we just need boundedness and some control in the
sense of Gubinelli. %
Then weak (measure-valued) forward existence gives regular backward
uniqueness. At last, we note that, subject to suitable smoothness
assumptions on the coefficients, regular forward equations can be viewed as
regular backward equations, from which we deduce regular forward uniqueness.

\bigskip

It is a natural question what the above RPDE\ solutions have to do with
classical SPDE solutions.  To this end, recall \cite[Ch.9]{FH14}
{consistency of RDEs with SDEs} in the following sense: RDE solutions
driven by $\mathbf{W}=\mathbf{W}^{\mathrm{Strato}}\left( \omega \right) $,
the usual (random) geometric rough path associated to Brownian motion $W$
via iterated Stratonovich integration are solutions to the corresponding
(Stratonovich)\ SDEs. Consider now - for the sake of argument - a regular
backward RPDE solution; that is, the unique solution $u=u\left( t,x;\mathbf{W%
}\right) $ to%
\begin{equation*}
-du_{t}=L[u_{t}]dt+\Gamma \lbrack u_{t}]d\mathbf{W}_{t}
\end{equation*}%
(with fixed $C_{b}^{2}$ time-$T$ terminal data). We expect that%
\begin{equation}
\tilde{u}\left( t,x\right) =\tilde{u}\left( t,x;\omega \right) =u\left( t,x;%
\mathbf{W}^{\mathrm{Strato}}\left( \omega \right) \right)
\label{utildeintro}
\end{equation}%
is also a (and hopefully: the unique) solution to the (backward)\ SPDE, again with fixed terminal data,%
\begin{equation*}
-d\tilde{u}_{t}=L[\tilde{u}_{t}]dt+\Gamma \lbrack \tilde{u}_{t}]\circ dW_{t}.
\end{equation*}%
(Similar for weak backward and weak/regular forward equations.)
Unfortunately, we cannot hope for a general RPDE/SPDE consistency statement
for the simple reason that the choice of spaces in which SPDE\ existence and
uniqueness statements are proven are model-dependent and therefore vary from paper to paper. In other words, checking that $\tilde{u}%
\left( t,x;\omega \right) $ is a - and then the (unique) - SPDE solution within a given
SPDE setting will necessarily require to check details specific to this setting. %will require to check
%different details depending on the chosen SPDE setting. 
Luckily, there are
arguments which do not force us into such a particular setting. 

\begin{itemize}
\item Consider a notion of (Stratonovich) SPDE\ solution for which there are
existence, uniqueness results and Wong--Zakai stability, by which we mean
that the (unique bounded, or finite-measure valued) solutions to the
random PDEs obtained by replacing $dW\left( \omega \right) $ by the mollified $%
\dot{W}^{\varepsilon }\left( \omega \right) dt$ converge to the unique SPDE
solution. (Such Wong--Zakai results are found e.g. in the works of Gy\"{o}%
ngy.) Assume also that our regularity assumptions fall
within the scope of these existence and uniqueness results. Then, for fixed
terminal (resp. initial) data, our unique RPDE\ solution, with driving rough
path $\mathbf{W}=\mathbf{W}^{\mathrm{Strato}}\left( \omega \right) $,
coincides with (and in fact, maybe a very pleasant version of) the unique
SPDE solution. (This follows immediately from continuous dependence of our
RPDE solutions on the driving rough paths, together with well-known rough
path convergence of mollifier approximations \cite{friz-victoir-book}.) In a
context of viscosity solutions, this argument was spelled out in \cite{FO14}.

\item Consider a notion of (Stratonovich) SPDE\ solution for which there are
existence, uniqueness results and a Feynman--Kac representation formula.
(This is the case in essentially every classical work on linear SPDEs, especially
in the filtering context.)
Recall that such SPDE\ Feynman--Kac formulas are conditional expectations,
given $W\left( \omega \right)$ (the observation, in the filtering context). In contrast, the Feynman--Kac
formula eluded to in Theorem \ref{thmmainloose}, is of unconditional form $%
\mathbb{E}^{t,x}\left( ...\right) $, the expectation taken over some hybrid It%
\^{o}-rough process (with rough driver $d\mathbf{W}$). By a stochastic
Fubini argument (similar to the one in \cite{DOR}) one can show that the Feynman--Kac
formula, evaluated at $\mathbf{W}=$ $\mathbf{W}^{\mathrm{Strato}}\left(
\omega \right) $, indeed yields the SPDE\ Feynman--Kac formula. In
particular, our unique RPDE\ solution, with driving rough path $\mathbf{W}=%
\mathbf{W}^{\mathrm{Strato}}\left( \omega \right) $, then coincides with the
unique SPDE solution.

\item At last, we consider an immediate consequence of our (rough path-)
wise definition in case of $\mathbf{W}=$ $\mathbf{W}^{\mathrm{Strato}}\left(
\omega \right) $. For the sake of argument, let us now focus on the weak
backward equation,%
\begin{equation*}
\left\langle u_{s},\phi \right\rangle -\left\langle u_{t},\phi \right\rangle
=\int_{s}^{t}\left\langle u_{r},L^{\ast }\phi \right\rangle
dr+\int_{s}^{t}\left\langle u_{r},\Gamma ^{\ast }\phi \right\rangle d\mathbf{%
W}_{r}.
\end{equation*}%
With $\tilde{u}\left( t,x;\omega \right) =u\left( t,x;\mathbf{W}^{\mathrm{%
Strato}}\left( \omega \right) \right) $, as before it follows from
consistency of rough with classical (backward) Stratonovich integration \cite%
[Ch.5]{FH14} that%
\begin{equation*}
\left\langle \tilde{u}_{s},\phi \right\rangle -\left\langle \tilde{u}%
_{t},\phi \right\rangle =\int_{s}^{t}\left\langle \tilde{u}_{r},L^{\ast
}\phi \right\rangle dr+\int_{s}^{t}\left\langle \tilde{u}_{r},\Gamma ^{\ast
}\phi \right\rangle \circ dW,
\end{equation*}%
for the same class of spatial test functions. Such notion of weak\ (or
distributional)\ SPDE solutions appear for instance in the works of Krylov,
e.g. \cite[Def. 4.6]{bib:krylovAnalytic}. Hence, whenever such a
notion of SPDE solution comes with uniqueness results, it is straight-forward
to see that $\tilde{u}$, i.e. our solution constructed via rough paths, must
coincide with the unique SPDE solution.
\end{itemize}

\subsection{Notation}

The second resp. first oder operators we shall consider are of the following form,
\begin{eqnarray*}
  Lu &:=& \frac{1}{2} \mathrm{Tr}\left( \sigma \left( x\right) \sigma ^{T}\left( x\right)
D^{2}u\right) +\left\langle b\left( x\right) ,Du\right\rangle +c\left(
x\right) u \\
\Gamma_{k}u &:=& \left\langle \beta _{k}\left( x\right) ,Du\right\rangle +\gamma
_{k}\left( x\right) u;
\end{eqnarray*}%
with $\sigma =\left( \sigma _{1},\dots ,\sigma _{\dbm}\right) ,\,\beta =\left(
\beta _{1},\dots ,\beta _{\drp}\right) $ and $b$ vector fields on $\mathbb{R}%
^{\dode} $ and scalar functions $c,\gamma _{1},\dots ,\gamma _{\drp}$.
We note that the formal adjoints are given as,
\begin{align*}
  L^* \varphi &= 
  \frac{1}{2} \mathrm{Tr}[ \tilde a(x) D^2 \varphi ]
  +
  \langle \tilde b(x), D\varphi \rangle
  +
  \tilde c(x) \varphi, \\
  \Gamma^*_k \varphi
  &= 
  \langle \tilde \beta_k (x), D\varphi \rangle
  +
  \tilde \gamma_k(x) \varphi
\end{align*}
where
\begin{align}
  \label{eq:adjointCoefficients}
  \begin{split}
    \tilde a(x) & := a(x) := \sigma\sigma^T (x) \\
    \tilde b_i(x) & := \partial_j a_{ji}(x) - b_i (x)\\
    \tilde c(x) & := \frac{1}{2} \partial_{ij} a_{ij}(x) - \operatorname{div}( b )(x) + c(x) \\
  \end{split}
  \begin{split}
    \tilde \beta_k (x) & := -\beta_k(x) \\
    \tilde \gamma_k (x)
    & :=
    -
    \operatorname{div}( \beta_k )(x)
    +
    \gamma_k(x)
  \end{split}
\end{align}

Precise assumptions on the coefficients will appear in the theorems below. Let us remark, however, that we did not push for optimal assumptions.
As is typical in rough path theory, $C_b^n$-regularity (bounded, with bounded derivatives up to order $n$) can often be improved to $\Lip^\gamma$-regularity with $\gamma \in (n-1,n)$, depending on the H\"older exponent of the driving rough path.

%some references to mention:
%
%\cite[Chapter 3]{bib:catellier}: only transport (i.e. $\sigma \equiv 0$) and constant vector field $\beta$
%      (so no rough path theory should be needed). but: he allows very irregular drift $b$ (regularization by noise)
%
%\cite{bib:gubinelliTindel}
%
%the new Jianfeng paper ..

\section{The backward equation}

Replacing the rough path by a smooth path, say $W\in C^{1}\left( \left[ 0,T\right], \mathbb{R}^{\drp}\right)$
we certainly want to recover a solution to the PDE 
\begin{align}
  \label{eq:preroughTVP}
  \begin{cases}
  -\partial _{t}u_t &= Lu_t + \Gamma_{k}u_{t} \dot{W}_{t}^{k} \\ 
  u\left( T,\cdot \right) &= g.
  \end{cases} 
\end{align}

For the precise statement of the following lemma, let us now introduce 
a suitable class of test functions with exponential decay, that will 
become important in the concept of weak solutions. 

\begin{definition}
  \label{def:expDecay}
  For $n \ge 0$ denote with $C^n_{\exp}(\R^\dode)$
  the class of functions $\phi \in C^n(\R^\dode)$ such that
  there exists $c > 0$ such that
  \begin{align*}
    |D^{k} \phi(x)|  \le c e^{-\frac{1}{c} |x|}, k = 0,1,\dots,n.
  \end{align*}
  Define the quasinorm%
  \footnote{ .. which we shall need in order to speak of bounded sets in $C^n_{\exp}(\R^\dode)$ .. }
  $||\cdot||_{C^n_{\exp}(\R^\dode)}$
  as the infimum over the values of $c$ satisfying the bound. Define moreover the space $C^{m,n}_{\exp}([0,T] \times \R^\dode)$
  to be the class of functions $\phi \in C^{m,n}([0,T]\times\R^\dode)$
  such that there exists $c > 0$ such that
  \begin{align*}
    |D^{j,k} \phi(t,x)|  \le c e^{-\frac{1}{c} |x|}, j=0,\dots,m, k = 0,1,\dots,n.
  \end{align*}
\end{definition}

We then recall the following Feynman-Kac representation for solutions to the  classical equation \eqref{eq:preroughTVP}. 

\begin{lemma}
  \label{lem:classicalFK}
  Assume
  $c,b,\sigma_i, \gamma_j,\beta_k \in C^2_b$, $i=1,\dots,\dbm$, $j,k=1,\dots,\drp$.
  Let $u$ be given as
  \begin{equation}
    u\left( t, x\right) =\E^{t,x}\left[ g\left( X_{T}\right) \exp \left(
    \int_{t}^{T}c\left( X_{r}\right) dr+\int_{t}^{T}\gamma \left( X_{r}\right) 
    \dot{W}_{r}dr\right) \right]  \label{preroughFKformula}
  \end{equation}%
  with
  \begin{equation*}
    dX_{t}=\sigma \left( X_{t}\right) dB\left( \omega \right) +b\left(
    X_{t}\right) dt+\beta \left( X_{t}\right) \dot{W}_{t}dt,
  \end{equation*}%
  where $B$ is a $\dbm$-dimensional Brownian motion and $W \in C^1( [0,T], \R^\drp )$.
  \begin{enumerate}[(i)]
    \item
      If $g \in C^2_b(\R^\dode)$ then $u$ is the unique $C_b^{1,2}([0,T]\times\R^\dode)$ solution to \eqref{eq:preroughTVP}.
      If moreover $g \in C^2_{\exp}(\R^\dode)$ then $u \in C^{1,2}_{\exp}([0,T]\times\R^\dode)$.

    \item 
      If $g \in C_b(\R^\dode)$ then
      $u \in C_b([0,T] \times \R^\dode)$ and it is the unique bounded analytically weak solution to \eqref{eq:preroughTVP},
      that is, for $\varphi \in \mathcal{D}(\R^\dode)$
      \begin{align}
        \label{eq:classicalWeakPDE}
        \left\langle u_{t},\varphi \right\rangle
        =
        \left\langle g,\varphi \right\rangle
        +
        \int_{t}^{T}\left\langle u_{r},L^{\ast }\varphi \right\rangle dr
        +
        \int_{t}^{T}\,\left\langle u_{r},\Gamma^{\ast }\varphi \right\rangle dW_{r}.
      \end{align}

  \end{enumerate}
\end{lemma}
\begin{proof}
  Let us first note that the expectation actually exists,
  since $g$, $c$, $\gamma$ and $|\dot W|$ are bounded.

  (i):
  The proof amounts to taking derivatives under the expectation,
  see for example Theorem V.7.4 in \cite{bib:krylovDiffusions}, which shows that $u$ is a $C^{1,2}_b$ solution.%
  \footnote{
    In \cite{bib:krylovDiffusions} it is assumed that the term in the exponential is non-positive, but a term bounded from below poses no additional difficulty: just replace $u(t,x)$ by $u(t,x) e^{- c (T-t)}$ for $c$ sufficiently large.}
  Uniqueness follows from the maximum principle, see for example Theorem 8.1.4 in \cite{bib:krylovHolder}.

  If  $g \in C^2_{\exp}(\R^\dode)$ then
  one can show that actually $u \in C^{1,2}_{\exp}([0,T]\times\R^\dode)$.
  This is similar to the rough case in Theorem \ref{thm:roughBackwardEquation}, so we omit the proof here.

  (ii):
  Take some $g^n \in C^2_b(\R^\dode)$ converging to $g$ locally uniformly,
  uniformly bounded by $2 ||g||_\infty$
  Let $u^n$ be the corresponding classical solution from part (i).
  Then $u^n$ satisfies \eqref{eq:classicalWeakPDE} with $g$ replaced by $g^n$.
  Now by the Feynman-Kac representation, we get for every $N > 0$,
  \begin{align*}
    |u^n(t,x) - u(t,x)|
    &\lesssim
    \E[ |g^n(X^{t,x}_T) - g(X^{t,x}_T)|^2 ]^{1/2} \\
    &\le
    \sup_{|y|\le N} |g^n(y) - g(y)|
    +
    2 ||g||_{\infty} \E[ 1_{ [-N,N]^C }( |X^{t,x}_T ) ].
  \end{align*}
  Hence for every $R > 0$
  \begin{align*}
    \sup_{|x| \le R} |u^n(t,x) - u(t,x)|
    \lesssim
    \sup_{|y|\le N} |g^n(y) - g(y)|
    +
    \frac{1}{N} \sup_{|x|\le R} \E[ |X^{t,x}_T| ],
  \end{align*}
  from which the locally uniform convergence of $u^n_t$ to $u_t$ follows, uniformly in $t \le T$.
  Taking the limit in the integral equation, we then see that $u$ satisfies \eqref{eq:classicalWeakPDE}.

  To show uniqueness,
  let $u \in C_b([0,T]\times\R^\dode)$ be any solution to \eqref{eq:classicalWeakPDE}.
  It is immediate that the equation then also holds for test functions $\varphi \in C^2_c(\R^\dode)$.
  It is straightforward %, via partition of the time interval,
  to show that for $\varphi \in C^{1,2}_c([0,T] \times \R^\dode)$ we have
  \begin{align}
    \label{eq:classicalWeakPDETimeDependent}
    \left\langle u_{t},\varphi_t \right\rangle
    =
    \left\langle g,\varphi_T \right\rangle
    +
    \int_{t}^{T}\left\langle u_{r}, -\partial_t \varphi_r + L^{\ast }\varphi_r \right\rangle dr
    +
    \int_{t}^{T}\,\left\langle u_{r},\Gamma^{\ast }\varphi_r \right\rangle dW_{r}.
  \end{align}
  Finally, via dominated convergence, \eqref{eq:classicalWeakPDETimeDependent} also holds for
  $\varphi \in C^{1,2}_{\exp}( [0,T] \times \R^\dode)$.

  Now Lemma \ref{lem:forwardEquationSmooth} (iv) gives us for every $t \in [0,T)$, $\phi \in C^4_{\exp}(\R^d)$
  such a $\varphi$ (on $[t,T]$) that satisfies
  \begin{align*}
    \partial_s \varphi_s &= L^{\ast} \varphi_s + \Gamma^\ast \varphi_s \dot{W}_s \\
    \varphi_t &= \phi.
  \end{align*}

  Then, by \eqref{eq:classicalWeakPDETimeDependent},
  \begin{align*}
    \langle u_t, \phi \rangle
    = 
    \langle u_t, \varphi_t \rangle
    =
    \langle g, \varphi_T \rangle.
  \end{align*}
  So, tested against $\phi \in C^4_c (\R^\dode)$, all solutions coincide at every $t \in [0,T]$, which gives uniqueness in $C_b([0,T]\times\R^\dode)$.
\end{proof}

When replacing $W$ by a rough path $\mathbf{W}$, we are formally interested in the following equation
\begin{equation}
  \label{eq:backwardRPDE}
  \begin{cases}
    -du & = Lu dt+ \Gamma_k u d\mathbf{W}^k \\
    u\left( T,\cdot \right) &=g.%
  \end{cases}.
\end{equation}

We will next introduce two solution concepts, weak and regular in nature 
(see Definitions \ref{def:weakSolutionRough} and \ref{def:strongSolutionRough} below). 

\begin{definition}[{\bf analytically weak backward RPDE solution}]
  \label{def:weakSolutionRough}
  Given an $\alpha $-H\"{o}lder rough path $\mathbf{W=}\left( W,\mathbb{W}%
\right)$, $\alpha \in (1/3,1/2]$, we say that a bounded, measurable function $u=u\left( t,x;\mathbf{W}\right)
  =u_{t}\left( x;\mathbf{W}\right) $ is an analytically weak solution to
  \eqref{eq:backwardRPDE},
  if for all functions
  $\varphi \in C^3_{\exp}(\R^\dode)$,
  we have $\left(Y^\varphi, (Y^\varphi)^\prime \right) \in \D_{W}^{2 \alpha}$
  with
  \begin{equation*}
    Y_{t}^\varphi := \left\langle u_{t},\Gamma^{\ast}_i \varphi \right\rangle \in \R^\drp ,\,\,\,
    (Y_{t}^\varphi)^{\prime} := -\left\langle u_{t},\Gamma^{\ast }_j \Gamma^{\ast }_i \varphi \right\rangle \in L(\R^\drp,\R^\drp),
  \end{equation*}
  that is
  \begin{align}
    \label{eq:weakSolutionRoughControlled}
    ||Y^\varphi,(Y^\varphi)'||_{W,\alpha} < \infty,
  \end{align}
  and the following equation is satisfied
  \begin{equation}
    \label{eq:defAWRPDEsol}
    \left\langle u_{t},\varphi \right\rangle
    =
    \left\langle g,\varphi \right\rangle
    +
    \int_{t}^{T}\left\langle u_{r},L^{\ast }\varphi \right\rangle dr
    +
    \int_{t}^{T}\,\left\langle u_{r},\Gamma^{\ast }\varphi \right\rangle d\mathbf{W}_{r}, \qquad 0 \le t \le T.
  \end{equation}%
  Here, $\int Yd\mathbf{W}$ is the rough integral against $(Y,Y')$.
\end{definition}

\begin{remark}
  Different from the smooth case, Lemma \ref{lem:classicalFK}, we work
  with test functions in the larger class $C^3_{\exp}$ here.
  This is necessary, since the presence of the rough integral makes it impossible
  to automatically enlarge the space of functions for which the integral equation holds,
  as was done in the proof of Lemma \ref{lem:classicalFK}.
\end{remark}

\begin{remark}
  Heuristically, the origin of the compensator term $Y_{t}^{\prime
  }=\left\langle u_{t},\Gamma^{\ast }\Gamma^{\ast }\varphi \right\rangle $ can be seen as follows.
  One certainly expects that%
  \begin{equation*}
    \int_{s}^{t}\,\left\langle u_r,\Gamma^{\ast }\varphi \right\rangle {d}\mathbf{W}_r
    \approx \left\langle u_{s},\Gamma^{\ast }\varphi \right\rangle W_{s,t}
  \end{equation*}%
  where $a\approx b$ means $a-b=O\left( \left\vert t-s\right\vert^{2\alpha
  }\right) $.
  Hence, in view of (\ref{eq:defAWRPDEsol}),%
  \begin{equation*}
    \left\langle u_{t},\varphi \right\rangle -\left\langle u_{s},\varphi \right\rangle
    \approx -\int_{s}^{t}\,\left\langle u, \Gamma^{\ast }\varphi\right\rangle {d}\mathbf{W}
    \approx - \left\langle u_{s},\Gamma^{\ast}\varphi \right\rangle W_{s,t}
  \end{equation*}
Replacing $\varphi$ by $\Gamma^\ast \varphi$ (note that the latter is not in $C^3_{\exp}$ though) gives 
  \begin{equation*}
    \left\langle u_{t},\Gamma^{\ast }\varphi \right\rangle -\left\langle u_{s},\Gamma^{\ast }\varphi \right\rangle
    =
    -\left\langle u_{s},\Gamma^{\ast }\Gamma^{\ast}\varphi \right\rangle W_{s,t}+O\left( \left\vert v-u\right\vert ^{2\alpha }\right)
  \end{equation*}%
  so that $t\mapsto \left\langle u_{t},\Gamma^{\ast }\varphi \right\rangle $ is
  controlled by $W$, with Gubinelli derivative $-\left\langle u_{t},\Gamma^{\ast}\Gamma^{\ast }\varphi \right\rangle$.
\end{remark}

\begin{definition}[{\bf regular backward RPDE solution}]   \label{def:strongSolutionRough}
  Given an $\alpha $-H\"{o}lder rough path $\mathbf{W=}\left( W,\mathbb{W}%
\right)$, $\alpha \in (1/3,1/2]$, we say that a function $u=u\left( t,x;\mathbf{W}\right)
  \in C^{0,2}$ (with respect to $t,x$) is a solution to
  \eqref{eq:backwardRPDE}
  if $( \Gamma_k u, \Gamma_j \Gamma_k u )$ is controlled by $W$ and
  \begin{align*}
    u(t,x)
    =
    g(x)
    +
    \int_t^T L u(r,x) dr
    +
    \int_t^T \Gamma_k u(r,x) d\mathbf{W}^k_r.
  \end{align*}
\end{definition} 

\begin{remark}
  If a regular solution in the sense of Definition \ref{def:strongSolutionRough} possesses
  a uniform bound on the control (see for example \eqref{eq:roughBackwardEquationUniformControl} below)
  then it is also a weak solution in the sense of Definition \ref{def:weakSolutionRough}.
\end{remark}

\begin{theorem}
  \label{thm:roughBackwardEquation}
  Throughout, $\mathbf{W}$ is a geometric $\alpha$-H\"older rough path, $\alpha \in (1/3,1/2]$.
  Assume $\sigma_i, \beta_j \in \Lip^{3}(\R^\dode), b \in \Lip^1(\R^\dode), c \in \Lip^1(\R^\dode)$,
  $\gamma_k \in \Lip^2(\R^\dode)$.
  Consider $g \in C^0_b(\R^\dode)$.

  \begin{enumerate}[(i)]
    \item {\bf Stability.}
  Let $u = u^W$ be the solution to \eqref{eq:preroughTVP}
  as given by the Feynman-Kac representation \eqref{preroughFKformula}, whenever $W\in C^{1}$.
  Pick $W^\epsilon \in C^1$ convergent 
  in rough path sense to $\mathbf{W}$. Then there exists a bounded, continuous function $u^\mathbf{W}$, independent of the choice of the approximating sequence, so that $u^{W^\epsilon} \to u^\mathbf{W}$ uniformly. 
  The resulting map $\mathbf{W} \mapsto u^\mathbf{W}$ is continuous.
  Moreover, the following Feynman-Kac representation holds, 
  \begin{align*}
    u^\mathbf{W}(t,x) = 
    \E^{t,x}\left[ g\left( X_T \right) \exp \left( \int_t^T c\left( X_r \right) dr +\int_t^T \gamma \left( X_r \right) d\mathbf{W}_r \right) \right],
  \end{align*}
  where $X$ solves the rough SDE (see Appendix, Lemma \ref{lem:roughSDE})
  \begin{equation}
    dX_{t}=\sigma \left( X_{t}\right) dB\left( \omega \right) +b\left(
    X_{t}\right) dt+\beta \left( X_{t}\right) d\mathbf{W}_t,
  \end{equation}%
  where $B$ is a $\dbm$-dimensional Brownian motion.

  \item {\bf Analytically weak backward RPDE solution.}
  Let $u=u^{\mathbf{W}}$
  be the function
  constructed in (i).
  Then  $u=u^{\mathbf{W}} \in C_b([0,T] \times \R^\dode)$ is a bounded solution to (\ref{eq:backwardRPDE}) in the sense of Definition \ref{def:weakSolutionRough}.
  Moreover, \eqref{eq:weakSolutionRoughControlled} is
  bounded, uniformly over bounded sets of $\varphi$ in 
  $C^3_{\exp} (\R^\dode)$, and it is the only solution in the class of $C_b$ functions $u$ satisfying \eqref{eq:weakSolutionRoughControlled}.

  \item {\bf Analytically regular backward RPDE solution.}
    \ifdraft 
    \todo[color=green]
    {
      because we need $4$ space derivatives, i.e. $3 + 3$ for RP VF,
      $1 + 3$ for drift VF, $2+4$ for RP integral, $0 + 4$ for drift integral.\\
      checking that $Du$ is controlled, $D^2g$ pops up as RP integrand
               and this needs to be $C^2_b$ ..
    }
    \fi
  Assume $\sigma_i, \beta_j \in \Lip^{6}(\R^\dode), b \in \Lip^{4}(\R^\dode), c \in \Lip^4(\R^\dode)$,
  $\gamma_k \in \Lip^6(\R^\dode)$
  and $g \in C^4_b(\R^\dode)$.
  Then $u=u^{\mathbf{W}} \in C^{0,4}_b( [0,T] \times \R^\dode )$ is a bounded solution to
  (\ref{eq:backwardRPDE}) in the sense of Definition \ref{def:strongSolutionRough}.
  It is the only solution in the class of functions in $C^{0,4}_b([0,T]\times\R^\dode)$ that satisfy
  \begin{align}
    \label{eq:roughBackwardEquationUniformControl}
    \begin{split}
    \sup_x || \Gamma u(\cdot,x), \Gamma \Gamma u(\cdot,x) ||_{W,\alpha} &< \infty.
    \end{split}
  \end{align}

  If moreover $g \in C^{4}_{\exp}(\R^\dode)$, then $u \in C^{0,4}_{\exp}([0,T]\times\R^\dode)$.
  \end{enumerate}
\end{theorem}
\begin{remark}
  We consider solutions in $C^{0,4}_b$, instead of the obvious choice $C^{0,2}_b$,
  because
  of two reasons. First, in order to show that $u$ is controlled by $W$
  we need $g \in C^4_b(\R^\dode)$ which automatically gives us $u \in C^{0,4}_b([0,T]\times\R^\dode)$.
  Second, this additional regularity is needed for the uniqueness proof via duality.
\end{remark}
\begin{remark}
  Results of the type in Theorem \ref{thm:roughBackwardEquation} (i), even in nonlinear situations, were obtained in 
  \cite{CF,CFO,DF12,CDFO13,FO14}. However, in all these references, the only
  intrinsic meaning of these equations was given in terms of a transformed
  equation, somewhat in the spirit of the Lions-Souganidis \cite{LSFully} theory of stochastic viscosity solutions. On the contrary, part (ii) and (iii) of the above theorem present a direct intrinsic characterization.
  See also \cite[Chapter 3]{bib:catellier}.

\end{remark}

\begin{proof}
  (i) This follows from stability of ``rough SDEs'', see Lemma \ref{lem:roughSDE}.

  (ii)
  \textbf{Existence}
  For simplicity only, we take $c=\gamma=b=0$ so that 
  \begin{eqnarray*}
    u\left( s,x\right) &=&\mathbb{E}\left[ g\left( X_{T}^{s,x}\right) \right] , 
    \notag \\
    dX_{t} &=&\sigma \left( X_{t}\right) dB_t\left( \omega \right) +\beta \left(
    X_{t}\right) d\mathbf{W}_t.
  \end{eqnarray*}%
  (With $X^{s,x}$ we mean the unique solution started at $X_{s}=x$.)
  In the following we consider the above SDE as an RDE w.r.t. the joint lift 
  ${\bf Z} = (Z, \mathbb Z)$ of $\bf W$ and the Brownian motion $B$ (see Lemma \ref{lem:jointLift} and Lemma \ref{lem:roughSDE} below). Denote with $\Phi$ its associated flow. 
  
  Recall $Y_t = \langle u_t, \bar \varphi \rangle$,
  $Y'_t = -\langle u_t, \Gamma^* \bar \varphi \rangle$, where $\bar \varphi := \Gamma^* \varphi$. Since $\varphi \in C^3_{\exp}$ and $\beta_j \in C^2_b$, $j=1,\dots,\drp$, we
  have that $\bar \varphi \in C^2_{\exp}$.
  Then
  \begin{align*}
    &Y_t - Y_s - Y'_t W_{s,t} \\
    &=
    \E\left[ \int_{\R^\dode} \left\{ g( \Phi_{t,T}(x) ) - g( \Phi_{s,T}(x) ) \right\} \bar \varphi(x)
    + g(\Phi_{t,T}(x)) \Gamma^* \bar \varphi( x ) W_{s,t} dx \right] \\
    &=
    \E\Bigl[ \int_{\R^\dode}
       g(y)
       \Bigl\{
       \bar \varphi( \Phi^{-1}_{t,T}(y) ) \det( D\Phi^{-1}_{t,T}(y) )
     - \bar \varphi( \Phi^{-1}_{s,T}(y) ) \det( D\Phi^{-1}_{s,T}(y) ) \\ 
      &\qquad \qquad \qquad 
      + \Gamma^* \bar \varphi( \Phi^{-1}_{t,T}(y) ) \det( D\Phi^{-1}_{t,T}(y)      ) W_{s,t} \Bigr\} dy \Bigr] \\
    &=
    \E\Bigl[ \int_{\R^\dode}
       g(y)
       \Bigl\{
       \bar \varphi( \Phi^{-1}_{t,T}(y) ) \det( D\Phi^{-1}_{t,T}(y) )
     - \bar \varphi( \Phi^{-1}_{s,T}(y) ) \det( D\Phi^{-1}_{s,T}(y) ) \\
     &\qquad \qquad \qquad
     + \Gamma^* \bar \varphi( \Phi^{-1}_{t,T}(y) ) \det( D\Phi^{-1}_{t,T}(y) ) W_{s,t}
     + \Gamma^* \bar \varphi( \Phi^{-1}_{t,T}(y) ) \det( D\Phi^{-1}_{t,T}(y) ) B_{s,t}
    \Bigr\} 
     dy
   \Bigr].
  \end{align*}
  Note that $\Gamma^{\ast}\varphi =-\operatorname{div}\left( b\varphi \right)$.
  Hence the term in curly brackets is bounded in absolute value, using Lemma \ref{lem:flowWithDeterminant}, by a constant times
  \begin{align*}
    ||\bar \varphi||_{C^3_b( M(y) )} \exp( C N_{1;[0,T]}(\mathbf{Z}) ) \left( ||\mathbf{Z}||_\beta + 1 \right)^{17 + 3 \dode} |t-s|^{2 \alpha}.
  \end{align*}
  Hence
  \begin{align*}
    |Y_t - Y_s - Y'_t W_{s,t}|
    \lesssim
    \int_{\R^\dode}
      \E\Bigl[ ||\varphi||_{C^3_b( M(y) )} \exp( C N_{1;[0,T]}(\mathbf{Z}) ) \left( ||\mathbf{Z}||_\beta + 1 \right)^{17 + 3 \dode} \Bigr]
      dy \, |t-s|^{2\alpha}\, . 
  \end{align*}
  Next observe that   
   \begin{align*}
    \E\Bigl[ ||\varphi||_{C^3_b( M(y) )} & \exp( C N_{1;[0,T]}(\mathbf{Z}) ) \left( ||\mathbf{Z}||_\beta + 1 \right)^{17 + 3 \dode} \Bigr]  \\
    & \le
    \E\Bigl[ ||\varphi||^3_{C^2_b( M(y) )} \Bigr]^{1/2} 
    \E\Bigl[ \exp( C N_{1;[0,T]}(\mathbf{Z}) )  
    \left( ||\mathbf{Z}||_\beta + 1 \right)^{17 + 3 \dode} \Bigr]^{1/2},
  \end{align*}
  and Lemma \ref{lem:jointLift} now implies that the last term is bounded 
  and Lemma \ref{lem:expDecaySurvivesFeynmanKacII} implies that the first 
  term decays exponentially in $y$. Therefore 
  \begin{align*}
    |Y_t - Y_s - Y'_t W_{s,t}| \lesssim |t-s|^{2\alpha},
  \end{align*}
  as desired.
 
  The estimate $|Y'_t - Y'_s| \le C |t-s|^{\alpha}$ is shown analogously,
  and then
  \begin{align*}
    |Y_t - Y_s - Y'_s W_{s,t}|
    \le
    |Y_t - Y_s - Y'_t W_{s,t}|
    +
    |Y'_{s,t} W_{s,t}|
    = O(|t-s|^{2\alpha}),
  \end{align*}
  as desired.

  It remains to show that the integral equation \eqref{eq:defAWRPDEsol} is satisfied.
  For this let $W^n$ be a sequence of smooth paths converging to $\mathbf{W}$ in $\alpha$-rough path metric.
  Let $u^n$ be the solution to \eqref{eq:classicalWeakPDE} as given by Lemma \ref{lem:classicalFK} (ii).

  Part (i) of the theorem now implies that $u^n$ converges locally uniformly to $u$, hence the convergence of all the terms in \eqref{eq:defAWRPDEsol}
  except the rough integral is immediate.
  For the rough integral, in view of Theorem 9.1 in \cite{FH14}, it is enough to show that
  \begin{align*}
    &\sup_n ||{Y'}^n||_{\alpha} < \infty
    \qquad \sup_n ||{Y'}^n - Y'||_\infty \to 0 \\
    &\sup_n ||R^n||_{2 \alpha} < \infty
    \qquad \sup_n ||R^n - R||_\infty \to 0,
  \end{align*}
  with $Y^n_t := \langle u^n_t, \Gamma^* \varphi \rangle, {Y'}^n := \langle u^n_t, \Gamma^* \Gamma^* \varphi \rangle$ and
  \begin{align*}
    R^n_{s,t} = \langle u^n_t - u^n_s, \Gamma^* \varphi \rangle - \langle u^n_s, \Gamma^* \Gamma^* \varphi \rangle W^n_{s,t}.
  \end{align*}
  The first two statements follow from the fact that the preceding considerations were uniform 
  for $\mathbf{W}$ bounded in rough path norm.
  Finally, convergence in supremum norm of ${Y'}^n_t - Y'_t = \langle u^n_t - u_t, \Gamma^* \Gamma^* \varphi \rangle$ 
  and $R^n_{s,t} - R_{s,t}$ follows from local uniform convergence of $u^n$.

  \textbf{Uniqueness}
  Let $\phi \in C^{0,3}_{\exp}([t,T],\R^\dode)$ be such that
  \begin{align*}
    \phi(t,x) = \varphi(x) + \int_t^T \alpha(r,x) dr + \int_t^T \eta_i(r,x) d\mathbf{W}^i_r,
  \end{align*}
  with $\eta \in C^{0,3}_{\exp}( [0,T] \times \R^\dode )$,
  and $(\eta_{i=1,\dots,\drp},\eta'_{i,j=1,\dots,\drp})$ controlled by $W$.
  Assume moreover
  for some $\delta > 0$
  \begin{align*}
    ||\eta(x),\eta'(x)||_{W,\alpha} \lesssim e^{-\delta |x|}, \\
    ||D\eta(x),D\eta'(x)||_{W,\alpha} \lesssim e^{-\delta |x|}.
  \end{align*}
  Then by Lemma \ref{lem:productRuleWeak}
  \begin{align*}
    \langle u_T, \phi_T \rangle
    =
    \langle u_t, \phi_t \rangle
    -
    \int_t^T \langle u_r, L^* \phi_r \rangle dr
    -
    \int_t^T \langle u_r, \Gamma^*_k \phi_r \rangle d\mathbf{W}^k_r
    +
    \int_t^T \langle u_r, \alpha(r) \rangle dr
    +
    \int_t^T \langle u_r, \eta_k(r) \rangle d\mathbf{W}^k_r.
  \end{align*}

  So it remains to find, for given $\varphi$, such a $\phi$
  with $\alpha(r) = L^* \phi(r)$,
  $\eta_i(r) = \Gamma^*_i \phi(r)$
  and $\eta'_{i,j}(r) = \Gamma^*_j \Gamma^*_i \phi(r)$.
  But this is exactly what Theorem \ref{thm:roughForwardEquation} (iii) gives us for
  $\varphi \in C^4_{\exp}(\R^\dode)$.
  Then
  \begin{align*}
    \langle g, \phi_T \rangle
    = \langle u_T, \phi_T \rangle
    = \langle u_t, \phi_t \rangle
    = \langle u_t, \varphi \rangle,
  \end{align*}
  which gives uniqueness of $u_t$. This holds for all $t \in [0,T]$, which gives uniqueness of $u$.

  ~\\

  ~\\
  (iii)
  Again, for simplicity only, we take $c=\gamma=b=0$ so that
  \begin{align*}
    u(t,x) = \E\left[ g\left( X^{t,x}_T \right) \right].
  \end{align*}

  Then
  \begin{align*}
    Du(t,x)
    =
    \E[ Dg( X^{t,x}_T ) D X^{t,x}_T ].
  \end{align*}
  indeed, using the integrability of $DX$ given by Lemma \ref{lem:roughSDE} and the fact that $Dg$ is bounded,
  the statement follows from interchanging differentiation and integration,
  see for example
  Theorem 8.1.2 in \cite{bib:friedlander}.

  Then by Lemma \ref{lem:flowLeftPoint}
  \begin{align*}
    &
    Du(t,x)
    -
    Du(s,x) \\
    &=
    \E[ Dg( X^{t,x}_T ) D X^{t,x}_T - Dg( X^{s,x}_T ) D X^{s,x}_T ] \\
    &=
    \E\left[ \int_s^t Dg( X^{r,x}_T ) dDX^{r,x}_T
        +
      \int_s^t  D^2 g( X^{r,x}_T ) \langle dX^{r,x}_T, DX^{r,x}_T \cdot \rangle  \right] \\
    &=
    \E\Bigl[ \int_s^t Dg( X^{r,x}_T ) D^2 X^{r,x}_T \langle V(x) dZ_r, \cdot \rangle
        +
        \int_s^t Dg( X^{r,x}_T ) DV(x) DX^{r,x}_T dZ_r \\
        &\qquad
        +
        \int_s^t D^2 g( X^{r,x}_T ) \langle DX^{r,x}_T V(x) dZ_r, DX^{r,x}_T \cdot \rangle \Bigr] + O(|t-s|) \\
    &=
    \E[ Dg(X^{s,x}_T ) D^2 X^{s,x}_T \langle \beta(x) W_{s,t}, \cdot \rangle ]
    +
    \E[ Dg(X^{s,x}_T ) D X^{s,x}_T D\beta(x) ] W_{s,t} \\
    &\qquad
    +
    \E[ D^2 g( X^{s,x}_T ) \langle DX^{s,x}_T V(x) W_{s,t}, DX^{s,x}_T \cdot \rangle ]
    +
    O(|t-s|^{2\alpha}) \\
    &=
    \partial_x[ \beta(x) Du(s,x) ] W_{s,t}
    +
    O(|t-s|^{2\alpha})
  \end{align*}
  So $\Gamma u$ is controlled as claimed and
  %the first point of
  \eqref{eq:roughBackwardEquationUniformControl} is satisfied.
  %while the second point follows analogously.
  %
  Showing that $u \in C^{0,4}_b$ also follows from differentiation under the expectation and
  the proof that the integral equation is satisfied now follows by using smooth approximations to $\mathbf{W}$, as in part (ii).

  Uniqueness follows from existence of the measure-valued forward equation.
  The argument is dual to the one that will be used in the proof of Theorem \ref{thm:roughForwardEquation} (ii), so we omit the proof here.

  Finally, the exponential decay of $u$, if $g \in C^4_{\exp}$, follows from Lemma \ref{lem:expDecaySurvivesFeynmanKac}. 
\end{proof}

\section{The forward equation}

\renewcommand{\L}{\mathcal{L}} % XXX
\renewcommand{\sp}{W} % smooth path XXX
\newcommand{\F}{\mathcal{F}}

We now consider the forward equation 
\begin{align}
  \label{eq:preroughIVP}
  \begin{cases}
    \partial_t \rho_t &= L^* \rho_t + \Gamma^*_k \rho_t \dot{W}^k_t \\
    \rho_0 &= p_0.
  \end{cases}
\end{align}
on the space $\mathcal{M}(\R^\dode)$ of finite measures on $\R^\dode$. 

Equation \eqref{eq:preroughIVP} is dual to the backward equation - considered in the previous section - in a sense that will be made precise in the following (see in particular 
Corollary \ref{cor:Duality} below).

The space $\mathcal{M}(\R^\dode)$ is endowed with the weak topology; that 
is $\mu_n \to \mu$ if $\mu_n( f ) \to \mu(f)$ for all $f \in C_b(\R^\dode)$. 
It is metrizable with compatible metric given by the Kantorovich-Rubinstein metric $d$, defined as 
\begin{align*}
  d(\mu,\nu) := \sup_{||f||_{\Lip^1(\R^\dode)} \le 1} 
  \left|\int_{\R^\dode} f(x) \nu(dx) - \int_{\R^\dode} f(x) \mu(dx)\right| 
\end{align*}
(see Chapter 8.3 in \cite{bib:bogachevII}). A compatible metric on the space of continuous finite-measure-valued paths is then given by 
$d(\mu_\cdot, \rho_\cdot) := \sup_{t\le T} d(\mu_t, \rho_t)$.

% Again, we have the following classical result:  
\begin{lemma}
  \label{lem:forwardEquationSmooth}
  Assume
  $c,b,\sigma_i, \gamma_j,\beta_k \in C^2_b$, $i=1,\dots,\dbm$, $j,k=1,\dots,\drp$. Define for $W \in C^1$ the measure valued process $\rho$ via its action on $f \in C_b(\R^\dode)$ as
  \begin{align}
    \label{eq:feynmanKacRho}
    \rho_t( f ) :=
    \E^{0,\nu}\left[ f\left( X_{t}\right) \exp \left( \int_{0}^{t}c\left( X_{s}\right) ds+\int_{0}^{s}\gamma \left( X_{s}\right) \dot{W}_{s}ds\right) \right],
  \end{align}
  where $\nu \in \mathcal{M}(\R^\dode)$ is the initial condition of the diffusion $X$ with dynamics
  \begin{equation*}
    dX_{t}=\sigma \left( X_{t}\right) dB\left( \omega \right) +b\left(
    X_{t}\right) dt+\beta \left( X_{t}\right) \dot{W}_{t}dt,
  \end{equation*}%
  where $B$ is a $\dbm$-dimensional Brownian motion.

  \begin{enumerate}[(i)]
  \item 

  Then $\rho$ is the unique, continuous
  $\mathcal{M}(\R^\dode)$-valued path which satisfies,
  for all $f \in C^2_b(\R^\dode)$,
  \begin{align}
    \label{eq:measureEvolution}
    \rho_t(f)
    =
    \nu(f)
    +
    \int_0^t \rho_s( L f ) ds
    +
    \int_0^t \rho_s( \Gamma_k f ) d\sp^k_s.
  \end{align}

    \item Assume moreover
      $\sigma_i \in C^4_b (\R^\dode )$, $i=1,\dots,\dbm$,
      $b, \beta_k \in C^3_b (\R^{\dode} )$, $k=1,\dots,\drp$.
      If $\nu$ has a density $p_0 \in C^2_{\exp}(\R^\dode)$ then
      $\rho_t$ has a density $p \in C^{1,2}_{\exp}([0,T]\times \R^\dode)$
      which is the unique bounded classical solution to \eqref{eq:preroughIVP}.
  \end{enumerate}

\end{lemma}

\begin{remark}
  We choose $p_0 \in C^2_{\exp}(\R^\dode)$ in part (ii) since 
  this is what we shall work with in the rough case.
  In the smooth case, the assumptions on the density $p_0$ can be weakened.
  Assume for example that $\nu$ has a density $p_0 \in C^2_b \cap L^1$. Then $\rho_t$ has a $C^2_b$
  density $p_t$ for all $t\ge 0$ and $p \in C^{1,2}_b$ is the unique bounded classical solution to \eqref{eq:preroughIVP}.
    Moreover,
    \begin{align}
      \label{eq:L1norm}
      ||p_t||_{L^1(\R^\dode)} = \rho_t(1) = \E^{0,\nu}\left[ \exp \left( \int_{0}^{t}c\left( X_{s}\right) ds+\int_{0}^{t}\gamma \left( X_{s}\right) \dot{W}_{s}ds\right) \right].
    \end{align}

  Indeed, by the smoothness assumptions on the coefficients,  \eqref{eq:preroughIVP} has a unique 
  solution in $C^{1,2}_b$ (this can again be seen by a Feynman-Kac argument, as in Lemma \ref{lem:classicalFK}).

  We have to show that the unique classical solution $p_t \in C_b^2$ of \eqref{eq:preroughIVP} with non-negative initial condition $p_0 \in C^2_b \cap L^1$
  is integrable. First recall that from the maximum principle, $p_t \ge 0$ for all $t\ge0$
  (see for example Theorem 8.1.4 in \cite{bib:krylovHolder}).
  Note that \eqref{eq:preroughIVP} implies that $\frac d{dt}\int \varphi p_t \, dx = \int {\tilde L}_t \varphi p_t \, dx$, 
  where ${\tilde L}_t \phi := L \phi + \Gamma_k \phi \dot{W}^k_t$, hence
  \begin{equation} 
    \label{eqErg1} 
    \int \varphi p_t\, dx = \int \varphi p_0\, dx + \int_0^t \int \tilde L_s \varphi p_s \, dx \, ds \,  
  \end{equation} 
  for any smooth and compactly supported function $\varphi$. Our aim now is to extend this equality to the constant function 
  $\varphi\equiv 1$. To this end consider for $\varepsilon > 0$ the function 
  $$ 
    \varphi_\varepsilon (x) := \varphi \left( \varepsilon\|x\|^2\right) \, , 
  $$ 
  where $\varphi (r) = (1+r)^{-\frac{e+1}2}$, $r \ge 0$. It is easy to check that both $\varphi_\varepsilon$ and 
  $$ 
  \begin{aligned} 
  L\varphi_\varepsilon (x) 
  & = - \varepsilon (e+1)\left( 1+\varepsilon \|x\|^2\right)^{-\frac{e+3}2} \left(\sum_{ij} (\sigma\sigma^T)_{ij} (x) 
  + \sum_i ({\tilde b}_i)_t (x) x_i \right)  + c\varphi_\varepsilon (x) \\ 
  & \qquad 
  + \varepsilon^2 (e+1)(e+3)\left( 1+\varepsilon \|x\|^2\right)^{-\frac{e+5}2} \left(\sum_{ij} (\sigma\sigma^T)_{ij} (x)x_i x_j\right)  
  \end{aligned} 
  $$ 
  are integrable. Since the coefficients $\sigma\sigma^T$ and $\tilde b_t := b + \Gamma_k \dot{W}^k_t$ have at most linear growth and $c$ is bounded, there exists a 
  finite constant $M$, independent of $\varepsilon$, such that   
  $$ 
    \tilde L_s \varphi_\varepsilon \le M \varphi_\varepsilon \, . 
  $$ 
  
  Next fix a smooth compactly supported test function $\chi$ on $\mathbb R$ satisfying $1_{[-1,1]}\le \chi\le 1_{[-2,2]}$ and let 
  $\chi_N (x) := \chi\left(\frac{\|x\|^2}{N^2}\right)$. Then $\chi_N\varphi_\varepsilon$ is compactly supported and 
  $$  
  \begin{aligned} 
  \tilde L_t \left( \chi_N\varphi_\varepsilon \right) 
  & = \chi_N \tilde L_t \varphi_\varepsilon - 4\varepsilon \frac{e+1}{N^2} \chi^\prime \left(\frac{\|x\|^2}{N^2}\right) 
  \left( 1+ \varepsilon\|x\|^2\right)^{-\frac{d+3}{2}} \sum_{ij}(\sigma\sigma^T) (x) x_ix_j \\
  & \qquad + ( (L_0)_t \chi_N)\varphi_\varepsilon 
  \end{aligned} 
  $$ 
  where $({\tilde L}_0)_t u = {\tilde L}_t u - cu$. Again due to the assumptions on the coefficients of $L$ (resp. $L_0$) we obtain that $L_0\chi_N$ is uniformly 
  bounded in $N$, so that $|\tilde{L}_t \left( \chi_N\varphi_\varepsilon \right)|$ is uniformly bounded in $N$ in terms of $\varphi_\varepsilon$ and 
  $|{\tilde L}_t\varphi_\varepsilon|$. Since $\tilde L_t \left( \chi_N\varphi_\varepsilon \right) \to \tilde L_t \varphi_\varepsilon$ pointwise, Lebesgue's dominated 
  convergence now implies that \eqref{eqErg1} extends to the limit $N\to\infty$, hence   
  
  \medskip 
  \begin{equation} 
  \label{eqErg2}
  \begin{aligned}  
  \int \varphi_\varepsilon p_t\, dx 
  & = \int \varphi_\varepsilon p_0\, dx + \int_0^t \int \tilde L_s \varphi_\varepsilon p_s \, dx \, ds \\ 
  & \le \int \varphi_\varepsilon p_0\, dx + M \int_0^t \int \varphi_\varepsilon p_s \, dx \, ds\, . 
  \end{aligned} 
  \end{equation} 
  Gronwall's lemma now implies that 
  $$
  \int \varphi_\varepsilon p_t \, dx \le e^{Mt} \int \varphi_\varepsilon p_0 \, dx\, . 
  $$ 
  Since $p_0$ is integrable, we can now take the limit $\varepsilon\downarrow 0$ to conclude with Fatou's lemma that 
  $$ 
  \int p_t \, dx \le e^{Mt} \int p_0\, dx < \infty\, . 
  $$ 
  
  Hence $\nu_t(f) := \int p_t(x) f(x) dx$ defines a finite-measure valued path
  and it satisfies \eqref{eq:measureEvolution}. By uniqueness it hence coincides with $\rho$.
  The expression for the $L^1$-norm of $p_t$ then follows from (i).
\end{remark}

\begin{proof}
  (i):
  Equation \eqref{eq:measureEvolution} is satisfied
  by an application of It\=o's formula,
  see for example Theorem 3.24 in \cite{bib:bainCrisan} for a similar argument.
  Uniqueness follows as in Theorem 4.16 in \cite{bib:bainCrisan}.
  Let us sketch the argument.%
  \footnote
  {
    Our setting here is simpler then in \cite{bib:bainCrisan}, since our coefficients and their derivatives are bounded.
  }
  First one shows that every solution to \eqref{eq:measureEvolution}
  also satisfies for $\varphi \in C^{1,2}_b([0,T], \R^\dode)$
  \begin{align}
    \label{eq:measureEvolutionTimeDependent}
    \rho_t(\varphi)
    =
    \nu(\varphi)
    +
    \int_0^t \rho_s( \partial_t \varphi + L \varphi ) ds
    +
    \int_0^t \rho_s( \Gamma_k \varphi ) d\sp^k_s.
  \end{align}
  Given now $\Phi \in C^\infty_b(\R^\dode)$ and $t \le T$ consider any
  solution $v \in C^{1,2}_b([0,t], \R^\dode)$ to the backward equation 
  \begin{align*}
    -\partial_t v &= L v + \Gamma_k v \dot{\sp}^k_t  \\
    v(t,\cdot) &= \Phi,
  \end{align*}
  the existence of which follows from Lemma \ref{lem:classicalFK}.

  Given two solutions $\rho, \bar \rho$ to \eqref{eq:measureEvolution}, we then have, by \eqref{eq:measureEvolutionTimeDependent},
  \begin{align*}
    \rho_t( \Phi )
    =
    \rho_t( v_t )
    =
    \rho_0( v_0 )
    =
    \bar \rho_0( v_0 )
    =
    \bar \rho_t( v_t )
    =
    \bar \rho_t( \Phi ),
  \end{align*}
  so $\rho_t, \bar \rho$ coincide on $C^\infty_b$.
  By pointwise uniformly bounded convergence they then also coincide on $C_b$, and hence $\rho_t = \bar \rho_t$ as desired.

  (ii):
  The coefficients of the dual equation, see \eqref{eq:adjointCoefficients}, all are in $C^2_b$.
  Hence using Lemma \ref{lem:classicalFK} there exists a unique bounded classical solution $p \in C^{1,2}_{\exp}([0,T]\times\R^\dode)$ to
  the PDE and a backward Feynman-Kac representation holds.
  Then it is in particular integrable and hence defines a measure-valued
  function $\mu$ on $[0,T]$ which satisfies \eqref{eq:measureEvolution}.
  By uniqueness for this equation it coincides with $\rho$.
\end{proof}

When replacing $W$ by a rough path $\mathbf{W}$, we are interested in the following equation
\begin{equation}
  \label{eq:forwardRPDE}
  \begin{cases}
    d \rho_t &= L^* \rho_t dt + \Gamma^*_k \rho_t d\mathbf{W}^k_t \\
    \rho_0 &= \nu.
  \end{cases}
\end{equation}

Two ways to make sense of this equation are given in the following definitions.

\begin{definition}[{\bf Measure valued forward RPDE solution}]
  \label{def:forwardSolutionRough}
  Given an $\alpha$-H\"older rough path $\mathbf{W} = (W,\mathbb{W})$, $\alpha \in (1/3,1/2]$,
  and $\nu \in \mathcal{M}(\R^\dode)$,
  we say that a continuous finite-measure-valued path $\rho_t$ is a weak solution to \eqref{eq:forwardRPDE}
  if for all 
  $f \in C^3_b(\R^\dode)$,
  $\rho_{t} \left( \Gamma_k f  \right)_{k=1,\dots,\drp}$
  is controlled by $W$ with Gubinelli derivative $\rho_t( \Gamma_j \Gamma_k f )_{k,j=1,\dots,\drp}$,
  that is
  \begin{align}
    \label{eq:forwardSolutionRoughControlled}
    || \rho_{\cdot} \left( \Gamma_k f  \right), \rho_\cdot( \Gamma_j \Gamma_k f ) ||_{W,\alpha} < \infty,
  \end{align}
  and the integral equation
  \begin{align}
    \label{eq:roughForwardEquation}  
    \rho_t(f) = \nu(f) + \int_0^t \rho_s( L f ) ds + \int_0^t \rho_s( \Gamma_k f ) d\mathbf{\sp}^k_s,
  \end{align}
  holds.
\end{definition}

\begin{definition}[{\bf Forward RPDE solution}]
  \label{def:strongForwardSolutionRough}
  Given an $\alpha$-H\"older rough path $\mathbf{W} = (W,\mathbb{W})$, $\alpha \in (1/3,1/2]$,
  and $\phi \in C^2_b(\R^\dode)$
  we say that $p \in C^{0,2}_b([0,T]\times \R^\dode)$ is a regular solution to
  \begin{align}
    \label{eq:strongForwardRough}
    \begin{cases}
    d p_t &= L^* p_t dt + \Gamma^*_k p_t d\mathbf{W}^k_t \\
      p_0 &= \phi,
    \end{cases}
  \end{align}
  if $\Gamma^*_k p$ is controlled by $W$ with
  Gubinelli derivative $\Gamma^*_j \Gamma^*_k p$ and the integral equation
  \begin{align*}
    p_t = \phi + \int_0^t L^* p_s ds + \int_0^t \Gamma^*_k p_s d\mathbf{W}^k_t,
  \end{align*}
  holds.
\end{definition}

\begin{theorem}
  \label{thm:roughForwardEquation}
  Throughout, $\mathbf{W}$ is a geometric $\alpha$-H\"older rough path, $\alpha\in (1/3,1/2]$.
  Assume $\sigma_i, \beta_j \in \Lip^{3}(\R^\dode), b \in \Lip^1(\R^\dode), c \in \Lip^1(\R^\dode)$,
  $\gamma_k \in \Lip^2(\R^\dode)$.
  Let $\nu$ be a finite measure.
  
  \begin{enumerate}[(i)]
    \item {\bf Stability.}
    Let $\rho = \rho^W$ be the solution to \eqref{eq:measureEvolution}
    as given by the Feynman-Kac representation \eqref{eq:feynmanKacRho}, whenever $W\in C^{1}$.
    Pick $W^\epsilon \in C^1$ convergent 
    in rough path sense to $\mathbf{W}$. Then there exists a continuous finite-measure-valued function $\rho^\mathbf{W}$, 
    independent of the choice of the approximating sequence, so that $d( \rho^{W^\epsilon},\rho^\mathbf{W}) \to 0$.
    The resulting map $\mathbf{W} \mapsto \rho^\mathbf{W}$ is continuous.
    Moreover, the following Feynman-Kac representation holds for $f \in C_b(\R^\dode)$
    \begin{align*}
      \rho^{\mathbf{W}}_t( f ) :=
        \E^{0,\rho_0}\left[ f\left( X_{t}\right) \exp \left( \int_{0}^{t}c\left( X_{s}\right) ds+\int_{0}^{t}\gamma \left( X_{s}\right) d\mathbf{W}_s \right) \right],
    \end{align*}
    where $X$ solves the same rough SDE as in Theorem \ref{thm:roughBackwardEquation}.

    \item {\bf Weak RPDE solution.}
    The measure-valued path
    $\rho^{\mathbf{W}}$ constructed in part (i)
    is a solution to \eqref{eq:forwardRPDE} in the sense of Definition \ref{def:forwardSolutionRough}.
    Moreover, \eqref{eq:forwardSolutionRoughControlled} is
    bounded, uniformly over bounded sets of $f$ in $C^3_b(\R^\dode)$.
    If the coefficients satisfy the stronger conditions of Theorem \ref{thm:roughBackwardEquation} (iii)
    then 
    $\rho^{\mathbf{W}}$ is the only solution in the class of measure-valued functions $\rho$ satisfying this uniform bound on \eqref{eq:forwardSolutionRoughControlled}.

    \item {\bf Regular RPDE solution.}
    Assume $\sigma_i \in \Lip^6(\R^\dode), \beta_j \in \Lip^7(\R^\dode), b \in \Lip^{5}(\R^\dode),
    \gamma_k \in \Lip^6(\R^\dode), c \in \Lip^4(\R^\dode)$.
    If $\rho_0$ has a density $p_0 \in C^4_{\exp}(\R^\dode)$, then $\rho_t$ has a density $p_t$
    for all times, and $p \in C^{0,4}_{\exp}([0,T]\times\R^d)$
    is a solution to \eqref{eq:strongForwardRough}
    in the sense of Definition \ref{def:strongForwardSolutionRough}.
    It is the only solution that in addition satisfies for some $\delta > 0$
    \begin{align*}
      || \Gamma^* u(\cdot,x), \Gamma^* \Gamma^* u(\cdot,x) ||_{W,\alpha} \lesssim e^{-\delta|x|}.
    \end{align*}
  \end{enumerate}
\end{theorem}

\begin{proof}
  (i):
  First of all we note that for fixed $f \in C_b(\R^{\dode})$ and fixed $t$ we have that 
  \begin{align*}
    \sp \mapsto \rho^{\sp}_t( f )
  \end{align*}
  is continuous in rough path topology.
  Indeed, this follows from Lemma \ref{lem:roughSDE} and is also seen to hold uniformly in $t$ and in bounded sets of $f$ in $C_b(\R^\dode)$.
  This also immediately gives the stated Feynman-Kac representation.

  (ii):
  Fix $f \in C^3_b(\R^\dode)$ and
  for simplicity take $b=\gamma=c=0$.
  Then note that
  \begin{align*}
    f(X_t)
    =  
    f(X_s)
    +
    \int_s^t L f(X_r) dr
    +
    \int_s^t \langle \sigma_i(X_r), Df(X_r) \rangle dB^i_r
    +
    \int_s^t \Gamma_i f(X_r) d\mathbf{\sp}^i_r.
  \end{align*}
  Taking expectation and applying Lemma \ref{lem:expectationControlII} we get
  \begin{align*}
    \rho_t(f) = \rho_s(f) 
    +
    \int_s^t \rho_r( L f ) dr
    +
    \int_s^t \rho_r( \Gamma_i f ) d\mathbf{W}^i_r,
  \end{align*}
  as well as the desired uniform bound on \eqref{eq:forwardSolutionRoughControlled}.

  To show uniqueness in part (ii),
  let $\phi \in C^{0,3}_b([0,t],\R^\dode)$ be 
  such that 
  \begin{align*}
    \phi(s,x) = \varphi(x) + \int_s^t \alpha(r,x) dr + \int_s^t \eta_i(r,x) d\mathbf{W}^i_r,
  \end{align*}
  for some $(\eta_{i=1,\dots,\drp},\eta'_{i,j=1,\dots,\drp})$ controlled by $W$,
  uniformly over $x$, i.e.
  \begin{align*}
    \sup_x \left[ ||\eta(x),\eta'(x)||_{W,\alpha} \right] &< \infty, \\
    \sup_x \left[ ||\Gamma_i \phi(x),\Gamma_i \eta(x)||_{W,\alpha} \right] &< \infty, \quad i=1,\dots,\drp.
  \end{align*}
  Moreover assume that $\eta \in C^{0,3}_b( [0,T] \times \R^\dode )$.
  Then by Lemma \ref{lem:productRuleMeasureValued}
  \begin{align*}
    \rho_t( \phi_t )
    =
    \rho_0( \phi_0 )
    +
    \int_0^t \rho_r( L \phi_r ) dr
    +
    \int_0^t \rho_r( \Gamma_k \phi_r ) d\mathbf{W}^k_r
    -
    \int_0^t \rho_r( \alpha(r) )dr
    -
    \int_0^t \rho_r( \eta_k(r) ) d\mathbf{W}^k_r.
  \end{align*}

  So it remains to find, for given $\varphi$, such a $\phi$
  with $\alpha(r) = L \phi(r)$,
  $\eta_i(r) = \Gamma_i \phi(r)$
  and $\eta'_{i,j}(r) = \Gamma_j \Gamma_i \phi(r)$.
  But this is exactly what Theorem \ref{thm:roughBackwardEquation} (iii) gives us for
  %\todo{
  %  this needs more regularity on coefficients:
  %   
  %}
  $\varphi \in C^4_b(\R^\dode)$.
  Then
  \begin{align*}
    \rho_t( \varphi ) = \rho_t( \phi_t ) = \rho_0( \phi_0 ),
  \end{align*}
  which gives uniqueness of $\rho$.

  (iii):
  The coefficients of the adjoint equation are given
  in \eqref{eq:adjointCoefficients}.
  In particular
  $\tilde \sigma_i \in \Lip^6(\R^\dode), \tilde \beta_j \in \Lip^6(\R^\dode), \tilde b \in \Lip^{4}(\R^\dode),
   \tilde c \in \Lip^4(\R^\dode), \tilde \gamma_k \in \Lip^6(\R^\dode)$.
  Hence the adjoint equation fits into the setting of 
  Theorem \ref{thm:roughBackwardEquation} (iii).
  In particular there exists a
  $C^{0,4}_b$ solution to \eqref{eq:strongForwardRough}
  and we can represent it as
  \begin{align*}
    p_t(x) = \E\left[ p_0( \tilde X^{T-t,x,\mathbf{W}}_T ) \exp\left( \int_{T-t}^T \tilde c( X^{T-t,x,\mathbf{W}}_r ) dr +  \int_{T-t}^T \tilde \gamma 
     (X^{T-t,x,\mathbf{W}}_r ) d\mathbf{W}_r \right) \right] ,
  \end{align*}
  for a rough SDE $\tilde X$.
  The exponential estimates on the control then follow by a similar argument
  as in the proof of Theorem \ref{thm:roughBackwardEquation} (iii),
  using Lemma \ref{lem:expDecaySurvivesFeynmanKacII} on the integrands $Dg, D^2 g$.
  Finally, $p_t$ is the density of $\rho_t$ of part (ii) due to the following reason: 
  $p_t$ is integrable because of the exponential decay, the corresponding measure satisfies \eqref{eq:strongForwardRough}, which by uniqueness for that equation then coincides with $\rho$. 
\end{proof}

The following lemma was needed in the previous proof.
\begin{lemma}
  \label{lem:expectationControlII}
  Assume $\sigma_i, \beta_j \in \Lip^3(\R^\dode)$, $b \in \Lip^1(\R^\dode)$.
  Let $\mathbf{Z}$ be the joint lift of a Brownian motion $B$ with a (deterministic) geometric $\alpha$-H\"older rough path $\mathbf{\sp}$, $\alpha \in (1/3,1/2]$
  (see Lemma \ref{lem:jointLift}).
  Let $X$ be the random RDE solution to
  \begin{align*}
    X_t
    &\ =\  X_0 + \int_0^t b(X_r) dr + \int_0^t \left(\sigma, \beta \right) d\mathbf{Z} \\
    &"=" X_0
    + \int_0^t b(X_r) dr 
    + \int_0^t \sigma_i(X_r) dB^i_r
    + \int_0^t \beta_i(X_r) d\mathbf{\sp}^i_r.
  \end{align*}
  Let $f \in C^3_b(\R^\dode)$ and define
  \begin{align*}
    (Y_t)_i &:= \E[ \beta_i^z(X_t) \partial_k f(X_t) ] \\
    (Y'_t)_{i,j} &:= \E[ \partial_k[ \partial_z f \beta_i^z ](X_r) \beta^k_j(X_r) ].
  \end{align*}

  Then $(Y,Y') \in \mathcal{D}_\sp^{2 \alpha}$ and
  \begin{align}
    \label{eq:RDEExpectation}
    \E[ f(X_t) ] = \E[ f(X_0) ]
    + \int_0^t \E[ bf (X_r) ] dr 
    + \int_0^t (Y,Y') d\mathbf{\sp}^i_r.
  \end{align}

  Moreover for all $R > 0$,
  \begin{align*}
    \sup_{||f||_{C^3_b} \le R} ||(Y,Y')||_{W,\alpha} < \infty. 
  \end{align*}
\end{lemma}
\begin{proof}
  For simplicity take $b=0$.
  First
  \begin{align*}
    X_{s,t} = \sigma_i(X_s) B^i_{s,t} + \beta_i(X_s) W^i_{s,t} + R_{s,t},
  \end{align*}
  where $||R||_{2\alpha} \le C \left( 1 + ||\mathbf{Z}||_{\alpha} \right)^3$ 
  (see Lemma \ref{lem:boundedRDEs}).

  Then,
  with $g_i := \beta^z_i \partial_z f$
  \begin{align*}
    g_i(X)_{s,t}
    =
    \sigma_j^k(X_s) \partial_k g_i(X_s) B^j_{s,t}
    +
    \beta_j^k(X_s) \partial_k g_i(X_s) W^j_{s,t}
    +
    R_{s,t},
  \end{align*}
  with $||\beta_j^k(X_s) \partial_k g(X_s)||_{\alpha} + ||R||_{2\alpha} \le C \left( 1 + ||\mathbf{Z}||_{\alpha} \right)^{3}$
  (see Lemma \ref{lem:applicationOfSmoothFunction}).

  Taking expectation and using integrability of $\mathbf{Z}$ (Lemma \ref{lem:jointLift}), we get
  \begin{align*}
    Y_{s,t} = Y'_s W_{s,t} + \bar R_{s,t},
  \end{align*}
  with $||Y'||_{\alpha} + ||\bar R||_{2\alpha} \le C < \infty$, as desired.

  Now for $W$ smooth, equation \eqref{eq:RDEExpectation} is satisfied by Fubini's theorem.
  Showing it for $\mathbf{W}$ a geometric rough path
  then follows via smooth approximations. This has for example already been
  done - in a similar setting - in the proof of Theorem \ref{thm:roughBackwardEquation}, so we omit the details here.
\end{proof}

The following result in the proof of the previous theorem is worth to be formulated separately.

\begin{corollary}[Duality]
\label{cor:Duality} 
  Assume the conditions of Theorem \ref{thm:roughBackwardEquation} (iii).
  Let $u$ be the unique solution to the backward equation \eqref{eq:defAWRPDEsol}
  and $\rho$ be the unique solution to the forward equation \eqref{eq:roughForwardEquation}.
  Then
  \begin{align*}
    \rho_t( u_t ) = \rho_0( u_0 ) \qquad \forall t \in [0,T].
  \end{align*}
\end{corollary}

\section{Appendix}

\newcommand{\arp}{X} % RP name in appendix

\subsection{Rough differential equations}

We recall the space of controlled paths.
\begin{definition}
  Let $\arp$ be an $\alpha$-H\"older continuous path, $\alpha \in (1/3,1/2]$.
  A path $Y \in C^\alpha$ is controlled by $\arp$ with derivative $Y' \in C^\alpha$
  (in short $(Y,Y') \in \mathcal{D}_W^{2\alpha}$),
  if
  \begin{align*}
    Y_{s,t} = Y'_s \arp_{s,t} + R^Y_{s,t},
  \end{align*}
  with $R^Y = O(|t-s|^{2\alpha})$. We use the following semi-norm on $\mathcal{D}_W^{2\alpha}$,
  \begin{align*}
    ||Y,Y'||_{\arp,\alpha} := ||Y'||_\alpha + ||R^Y||_{2 \alpha}.
  \end{align*}
\end{definition}

\begin{definition}
  Let $\omega $ be a control function (see Definition 1.6 in \cite{friz-victoir-book}). For $a >0$ and $\left[ s,t\right] \subset %
  \left[ 0,T \right] $ we set%
  \begin{eqnarray*}
  \tau _{0}\left( a \right) &=&s \\
  \tau _{i+1}\left( a \right) &=&\inf \left\{ u:\omega \left( \tau
  _{i},u\right) \geq a ,\tau _{i}\left( a \right) <u\leq t\right\}
  \wedge t
  \end{eqnarray*}%
  and define%
  \begin{equation*}
  N_{a ,\left[ s,t\right] }\left( \omega \right) =\sup \left\{ n\in 
  \mathbb{N\cup }\left\{ 0\right\} :\tau _{n}\left( a \right) <t\right\} .
  \end{equation*}%
  When $\omega $ arises from a (homogeneous) $p$-variation norm of a ($p$%
  -rough) path, such as $\omega _{\mathbf{\arp}}=\left\Vert \mathbf{\arp}\right\Vert
  _{p\text{-var;}\left[ \cdot ,\cdot \right] }^{p}$, we shall also write%
  \begin{equation*}
  N_{a ,\left[ s,t\right] }\left( \mathbf{\arp}\right) :=N_{a ,\left[
  s,t\right] }\left( \omega _{\mathbf{\arp}}\right).
  \end{equation*}
\end{definition}

\begin{remark}
  The importance of $N_{a;[0,T]}(\mathbf{\arp})$ stems from the fact that it has - contrary to the $p$-variation norm $||\mathbf{\arp}||_{p-var}$ - Gaussian
  integrability if $\mathbf{\arp} = \mathbf{B}$, the lift of Brownian motion (see \cite{CLL13,bib:frizRiedel}),
  or if $\mathbf{\arp} = \mathbf{Z}$, the joint lift of Brownian motion and a deterministic
  rough path used in the proof of Theorem \ref{thm:roughBackwardEquation}
  (see Lemma \ref{lem:jointLift} (iii)).
\end{remark}

\begin{lemma}[Bounded vector fields]
  \label{lem:boundedRDEs}
  
  Let $\mathbf{X}$ be a geometric $\alpha$-H\"older rough path, $\alpha \in (1/3,1/2]$.
  Let
  \begin{align*}
    dY = V(Y) d\mathbf{X},
  \end{align*}
  where $V=\left( V_{i}\right) _{i=1,\ldots ,\drp}$ is a collection of $\operatorname{Lip}^3(\R^\dode)$
  vector fields.

  Then with $(Y,Y') := (Y, V(Y))$ we have
  \begin{align}
    \label{eq:controlBoundBadVF}
    ||Y,Y'||_{X,\alpha} \le C \left( 1 + ||\mathbf{X}||_\alpha \right)^3.
  \end{align}

  Also
  \begin{align}
    \label{eq:infinityBoundBoundedVF}
    ||Y||_{1/\alpha-\mathrm{var}} \le C \left( 1 + N_{1;[0,T]}( \mathbf{X} ) \right).
  \end{align}
\end{lemma}
\begin{proof}
  \eqref{eq:controlBoundBadVF} follows from {\cite[Proposition 8.3]{FH14}} and
  \eqref{eq:infinityBoundBoundedVF} follows from \cite[Lemma 4, Corollary 3]{bib:frizRiedel}.
\end{proof}

\begin{lemma}[Linear vector fields]
  \label{lem:linearRDEs}
  Let $\mathbf{X}$ be a geometric $\alpha$-H\"older rough path, $\alpha \in (1/3,1/2]$.
  Let
  \begin{align*}
    dY = V(Y) d\mathbf{X},
  \end{align*}
  where $V=\left( V_{i}\right) _{i=1,\ldots ,\drp}$ is a collection of linear
  vector fields of the form $V_{i}\left( z\right) =A_{i}z+b_{i}$, where $A_{i}$ are
  $\dode \times \dode$ matrices and $b_{i}\in \mathbb{R}^{\dode}$.
  Then for $0\le s \le t \le T$:

  \begin{enumerate}
    \item
  \begin{align*}
    ||Y||_{p-var;[s,t]}
    \le
    C \left( 1 + |y_0| \right)
    ||\mathbf{X}||_{p-var;[s,t]}
    \exp( C N_{1;[0,T]}(\mathbf{X}) ).
  \end{align*}
  with $p:= 1/\alpha$, which implies
  \begin{align*}
    |Y|_{\alpha;[0,T]}
    \le
    C \left( 1 + |y_0| \right)
    ||\mathbf{X}||_{\alpha;[0,T]}
    \exp( C N_{1;[0,T]}(\mathbf{X}) ).
  \end{align*}

    \item
  \begin{align*}
    &|Y_{s,t} - V(Y_s) X_{s,t} - DV(Y_s) V(Y_s) \mathbb{X}_{s,}|
    \le
    C \exp( C N_1 ) ||\mathbf{X}||_{p-var;[s,t]}^3,
  \end{align*}
  which means that with $(Y,Y') := (Y, V(Y))$ we have
  \begin{align*}
    ||Y,Y'||_{X,\alpha} \le 
    C \exp( C N_1 ) \Bigl( ||\mathbf{X}||_{\alpha;[0,T]}^2 \vee ||\mathbf{X}||_{\alpha;[0,T]}^3 \Bigr).
  \end{align*}
  \end{enumerate}
\end{lemma}

\begin{proof}
  1.
  In what follows $C$ is a constant that can change from line to line.

  From \cite[Theorem 10.53]{friz-victoir-book} we have for any $s \le u \le v \le t$:
  \begin{align*}
    ||Y_{u,v}||
    \le
    C \left( 1 + |Y_u| \right)
    ||X||_{p-var;[u,v]}
    \exp( c ||X||_{p-var;[u,v]}^p ).
  \end{align*}
  Then, using $||Y_{u,v}|| = d(Y_u,Y_v) \ge d(Y_s,Y_v) - d(Y_s,Y_u) = ||Y_{s,v}|| - ||Y_{s,u}||$, we have
  \begin{align}
    ||Y_{s,v}||
    &\le
    C \left( 1 + |Y_u| \right) ||X||_{p-var;[u,v]} \exp( C ||X||_{p-var;[u,v]}^p )
    +
    ||Y_{s,u}|| \notag \\
    \label{eq:gronwallFirstStep}
    &\le
    C \left( 1 + |Y_s| + ||Y_{s,u}|| \right) ||X||_{p-var;[u,v]} \exp( C ||X||_{p-var;[u,v]}^p )
    +
    ||Y_{s,u}||.
  \end{align}

  On the one hand, this gives
  \begin{align*}
    ||Y_{s,v}||
    &\le
    C \left( 1 + |Y_s| + ||Y_{s,u}|| \right) \exp( C ||X||_{p-var;[u,v]}^p ).
  \end{align*}
  Now letting $s=\tau_0 < \dots < \tau_M < \tau_{M+1} = v \le t$, by induction,
  \begin{align*}
    ||Y_{s,v}||
    &\le C^{M+1} \left( (M+1) (1 + |Y_s|) \right) \exp( C \sum_{i=0}^M ||X||_{p-var;[\tau_i,\tau_{i+1}]}^p ) \\
    &\le C^{M+1} \left( 1 + |Y_s| \right) \exp( C \sum_{i=0}^M ||X||_{p-var;[\tau_i,\tau_{i+1}]}^p ).
  \end{align*}
  Hence
  \begin{align*}
    \sup_{u \in [s,t]} ||Y_{s,u}|| \le C \left( 1 + |Y_s| \right) \exp( C N_{1;[s,t]} ).
  \end{align*}

  Then, using again \eqref{eq:gronwallFirstStep},
  \begin{align*}
    &||Y_{s,v}|| \\
    &\le
    C \left( 1 + |Y_s| + ||Y_{s,u}|| \right) ||X||_{p-var;[u,v]} \exp( C ||X||_{p-var;[u,v]}^p )
    +
    ||Y_{s,u}|| \\
    &\le
    C \left( 1 + |Y_s| + C \left( 1 + |Y_s| \right) \exp( C N_{1;[s,t]} ) \right) ||X||_{p-var;[u,v]} \exp( C ||X||_{p-var;[u,v]}^p )
    +
    ||Y_{s,u}|| \\
    &\le
    C^2 2 \left( 1 + |Y_s| \right) ||X||_{p-var;[u,v]} \exp( C ||X||_{p-var;[u,v]}^p ) \exp( C N_{1;[s,t]} )
    +
    ||Y_{s,u}||.
  \end{align*}
  Then letting $s=\tau_0 < \dots < \tau_M < \tau_{M+1} = v \le t$, by induction,
  \begin{align*}
    &||Y_{s,v}|| \\
    &\le
    \sum_{i=0}^M 
    C^2 2 \left( 1 + |Y_s| \right) ||X||_{p-var;[\tau_i,\tau_{i+1}]} \exp( C ||X||_{p-var;[\tau_i,\tau_{i+1}]}^p ) \exp( C N_{1;[s,t]} ) \\
    &\le
    \sum_{i=0}^M 
    C^2 2 \left( 1 + |Y_s| \right) ||X||_{p-var;[s,v]} \exp( C ||X||_{p-var;[\tau_i,\tau_{i+1}]}^p ) \exp( C N_{1;[s,t]} ) \\
    &\le
    (M + 1) C
    \left( 1 + |Y_s| \right) ||X||_{p-var;[s,v]} \exp( C \max_i ||X||_{p-var;[\tau_i,\tau_{i+1}]}^p ) \exp( C N_{1;[s,t]} ).
  \end{align*}
  Then
  \begin{align*}
    ||Y_{s,t}||
    &\le
    (N_{1;[s,t]} + 1 ) C \left( 1 + |Y_s| \right) \exp( C N_{1;[s,t]} ) ||X||_{p-var;[s,t]} \\
    &\le
    C \left( 1 + |Y_0| \right) \exp( C N_{1;[0,T]} ) ||X||_{p-var;[s,t]},
  \end{align*}
  as desired.

  2. It is straightforward to construct
  $\operatorname{Lip}^3$ vector fields $\tilde V_i$, that coincide with $V_i$ on
  an open neighborhood of $Y$ and they can be chosen in such a way that
  \begin{align*}
    ||\tilde V||_{\operatorname{Lip}^3}
    &\le \max_i \left( |A_i| + |b_i| \right) \left( |Y|_\infty + 2 \right) 
    \le C \exp( C N_1 ).
  \end{align*}
  The first statement then follows from \cite[Corollary 10.15]{friz-victoir-book},
  which also yields the desired bound on $||R^Y||_{2\alpha}$.
  The bound on $||Y'||_\alpha$ follows from Step 1.
\end{proof}

\begin{lemma}
  \label{lem:flowWithDeterminant}
  Let $\varphi \in C^2_b(\R^\dode, \R)$,
  $\mathbf{X}$ a geometric $\beta$-H\"older rough path in $\R^\drp$, $\beta \in (1/3,1/2]$,
  $V_1, \dots, V_\drp \in \Lip^3(\R^\dode)$,
  and $\Psi$ the flow
  to the RDE
  \begin{align*}
    dY = V(Y) d\mathbf{X}.
  \end{align*}

  Then
  \begin{align}
    \begin{split}
    &\Bigl| \varphi( \Psi^{-1}_{t,T} ) \det( D\Psi^{-1}_{t,T} )
    -
    \varphi( \Psi^{-1}_{s,T} ) \det( D\Psi^{-1}_{s,T} ) \Bigr| \\
    &\qquad \le C
    ||\varphi||_{C^2_b( M(y) )}
    \exp( C N_{1;[0,T]}(\mathbf{X}) ) \left( ||\mathbf{X}||_\beta + 1 \right)^{17 + 3 d} |t-s|^{\beta},
    \end{split}
    \notag \\
    \label{eq:flowWithDeterminantToShow2}
    \begin{split}
     &\Bigl| \varphi( \Psi^{-1}_{t,T} ) \det( D\Psi^{-1}_{t,T} )
     - 
     \varphi( \Psi^{-1}_{s,T} ) \det( D\Psi^{-1}_{s,T} )
     -
     \operatorname{div}( \varphi V )( \Psi^{-1}_{t,T} )
     \det( D\Psi^{-1}_{t,T} )
     X_{s,t} \Bigr| \\
     &\qquad \le
     C
     ||\varphi||_{C^2_b( M(y) )}
     \exp( C N_{1;[0,T]}(\mathbf{X}) ) \left( ||\mathbf{X}||_\beta + 1 \right)^{17 + 3 d} |t-s|^{2 \beta},
    \end{split}
  \end{align}
  with $C = C(\beta, V,\varphi)$. Here the inverse flow and its Jacobian are evaluated at $y \in \R^\dode$.
  Moreover we used
  \begin{align*}
    M(y) &:= \{ x: \inf_{r \in [0,T]} |\Psi^{-1}_{T-r,T}(y)| - 1 \le |x| \le \sup_{r \in [0,T]} |\Psi^{-1}_{T-r,T}(y)| + 1 \}. % \\
  \end{align*}
\end{lemma}
\begin{proof}
  We shall need the fact that
  the inverse flow and its Jacobian satisfy the following RDEs
  (see for example \cite[Section 11]{friz-victoir-book}),
  \begin{align}
    \label{eq:RDEForReverseFlow}
    &d\Psi^{-1}_{T-\cdot,T}(y) = V( \Psi^{-1}_{T-\cdot,T}(y) ) dX_{T-\cdot}, &\ & \Psi^{-1}_{T,T}(y) = y, \\
    \label{eq:RDEForReverseFlowJacobian}
    &d D\Psi^{-1}_{T-\cdot,T}(y) = DV_i( \Psi^{-1}_{T-\cdot,T}(y) ) D\Psi^{-1}_{T-\cdot,T}(y) dX^i_{T-\cdot}, &\ & D\Psi^{-1}_{T,T}(y) = I,
  \end{align}

  We proceed to show the second inequality of the statement, as the first one follows analogously.
  In what follows $C$ will denote a constant changing from line to line, only depending on $\beta$, $V$, $\varphi$
  (but not on $X$ or $y$).

  Let
  $A_r := \varphi( \Psi^{-1}_{T-r,T}(y) )$,
  $B_r := \det( D\Psi^{-1}_{T-r,T}(y) )$.
  Using \eqref{eq:RDEForReverseFlow}
  we have that $(A,A') \in \D^{2\beta}_{\overleftarrow X}$, with
  \begin{align*}
    A'_r = \langle D\varphi( \Psi^{-1}_{T-r,T}(y) ), V( \Psi^{-1}_{T-r,T}(y) ) \rangle
  \end{align*}
  and $\overleftarrow X_r := X_{T-r}$.
  More specific, by Lemma \ref{lem:boundedRDEs}
  together with Lemma \ref{lem:applicationOfSmoothFunction}
  \begin{align*}
    ||A||_{\overleftarrow{X},\beta} \le C
    ||\varphi||_{C^2_b( M(y) ) }
    \left( ||\mathbf{X}||_\alpha + 1 \right)^8.
  \end{align*}
  where $M(y) := \{ x: \inf_{r \in [0,T]} |\Psi^{-1}_{T-r,T}(y)| - 1 \le |x| \le \sup_{r \in [0,T]} |\Psi^{-1}_{T-r,T}(y)| + 1 \}$.

  Moreover using \eqref{eq:RDEForReverseFlowJacobian}
  and the derivative of the determinant,
  \begin{align}
    \label{eq:derivativeOfDet}
    D \det|_A \cdot M = \det(A) \operatorname{Tr}[A^{-1} M],
  \end{align}
  we get that $(B,B') \in \D^{2\beta}_{\overleftarrow X}$, with
  \begin{align*}
    B'_r = (\operatorname{div} V)( \Psi^{-1}_{T-r,T}(y) ) \det( D\Psi^{-1}_{T-r,T}(y) ).
  \end{align*}
  More specific, by Lemma \ref{lem:linearRDEs}.2 together with Lemma \ref{lem:applicationOfSmoothFunction}
  \begin{align*}
    ||B||_{\overleftarrow{X},\beta}
    \le
    C
    ||\det||_{C^2_b( N(y) ) }
    \exp( C N_{1;[0,T] } ) \left( ||\mathbf{X}||_\alpha + 1 \right)^8 \\
    \le
    C
    \left( 1 + 
    \sup_{r \in [0,T]} |D\Psi^{-1}_{T-r,T}(y)| \right)^\dode
    \exp( C N_{1;[0,T] } ) \left( ||\mathbf{X}||_\alpha + 1 \right)^8 \\
  \end{align*}
  where $N(y) := \{ A : |A| \le \sup_{r \in [0,T]} |D\Psi^{-1}_{T-r,T}(y)| + 1 \}$.

  Applying Lemma \ref{lem:productOfControlledPaths}, we get for $0 \le u < v \le T$
  \begin{align*}
    &\left| A_v B_v - A_u B_u - \left( A'_u B_u - A_u B'_u \right) {\overleftarrow X}_{u,v} \right| \\
    &\le
    C
    \left( 1 + ||\mathbf{X}||_\alpha \right)
    \left(
    |\varphi(y)|
    +
    ||\varphi||_{C^2_b( M(y) ) }
    \left( ||\mathbf{X}||_\alpha + 1 \right)^8
    \right) \\
    &\qquad\times
    \left(
    1 
    +
    \left( 1 + 
    \sup_{r \in [0,T]} |D\Psi^{-1}_{T-r,T}(y)| \right)^\dode
    \exp( C N_{1;[0,T] } ) \left( ||\mathbf{X}||_\alpha + 1 \right)^8
    \right),
  \end{align*}
  Finally noting
  \begin{align*}
    A'_r B_r + A_r B_r' = - \operatorname{div}( \varphi V )( \Psi^{-1}_{T-r,T}(y) ) \det( D\varphi( \Psi^{-1}_{T-r,T}(y) ),
  \end{align*}
  and using $t = T-u$, $s = T-t$, the desired result follows from Lemma \ref{lem:forwardBackwardControlled}.
\end{proof}

The following result from the previous proof is worth noting separately.
\newcommand{\tr}{\operatorname{tr}}
\begin{lemma}[Liouville's formula for RDEs]
  Let $\mathbf{X}$ be a matrix-valued, geometric $\alpha$-H\"older rough path, $\alpha \in (0,1]$
  and consider the matrix-valued linear equation
  \begin{align*}
    dM_t &= d\mathbf{\arp}^i_t \cdot M_t \\
    M_0 &= I \in \R^{\dode \times \dode}.
  \end{align*}

  Denote $D_t := \det(M_t)$, then
  \begin{align*}
    dD_t &= 
    D_t \operatorname{tr}[d Z^i_t], \\
    D_0 &= 1,
  \end{align*}
  which is explicitly solved as
  \begin{align*}
    D_t = \exp( \tr[ \arp^i_t ] - \tr[ \arp^i_0 ] ).
  \end{align*}
\end{lemma}

\begin{lemma}
  \label{lem:productOfControlledPaths}
  Let $X \in \Cr^\alpha$, and $(A,A'),(B,B') \in \D^{2\alpha}_X$.
  Then $(Y,Y') := (A B, A'B + A B') \in \D^{2 \alpha}_X$
  and
  \begin{align*}
    &||Y'||_{\alpha} 
    +
    ||R^Y||_{2 \alpha}
    \le
    C \left( 1 + ||X||_\alpha \right)
    \left( |A_0| + ||A,A'||_{X,\alpha} \right)
    \left( |B_0| + ||B,B'||_{X,\alpha} \right).
  \end{align*}
\end{lemma}
\begin{proof}
  Straightforward calculation. 
\end{proof}

\begin{lemma}
  \label{lem:forwardBackwardControlled}
  Let $\tilde Y_t := Y_{T-t}, \tilde{Y}'_t := Y'_{T-t}$, $\tilde X_t := X_{T-t}$.
  If $(\tilde Y,\tilde{Y}') \in \D^{2\alpha}_{\tilde X}$, then
  $(Y,Y') \in \D^{2\alpha}_X$ and
  \begin{align*}
    ||Y'||_\alpha + ||R^Y||_{2\alpha}
    \le
    ||\tilde{Y}'||_\alpha + ||R^{\tilde Y}||_{2\alpha}
    +
    ||\tilde{Y}'||_\alpha ||X||_\alpha.
  \end{align*}
\end{lemma}
\begin{proof}
  This follows from
  \begin{align*}
    Y_t - Y_s - Y'_s X_{s,t}
    &=
    Y_t - Y_s - Y'_t X_{s,t} + Y'_{s,t} X_{s,t} \\
    &=
    \tilde Y_u - \tilde Y_v - \tilde Y'_u X_{v,u} + Y'_{s,t} X_{s,t}  \\
    &=
    - \left( \tilde Y_v - \tilde Y_u - \tilde Y'_u X_{u,v} \right) + Y'_{s,t} X_{s,t},
  \end{align*}
  where $v:=T-s$, $u:=T-t$.
\end{proof}

\begin{lemma}
  \label{lem:applicationOfSmoothFunction}
  Let $X$ be an $\alpha$-H\"older path, $\alpha \in (1/3,1/2]$.
  Let $(Y,Y') \in \D^{2\alpha}_X$,
  $\phi \in C^2_b$.
  Then $(\phi(Y), D\phi(Y) Y') \in \D^{2\alpha}_X$ with
  \begin{align*}
    ||\phi(Y), D\phi(Y) Y'||_{X,\alpha}
    \le
    C(\alpha,T)
    ||\phi||_{C^2_b}
    \left( 1 + ||X||_\alpha \right)^2
    \left( 1 + |Y_0'| + ||Y,Y'||_{X,\alpha} \right)^2.
  \end{align*}
\end{lemma}
\begin{proof}
  See \cite[Lemma 7.3]{FH14}. 
\end{proof}

\begin{lemma}[Adjoint equation]
  \label{lem:flowLeftPoint}
  Let $\mathbf{\arp}$ be a geometric $\alpha$-H\"older rough path, $\alpha \in (1/3,1/2]$.
  Let $V=\left( V_{i}\right) _{i=1,\ldots ,\drp}$ be a collection of $\operatorname{Lip}^3(\R^\dode)$
  vector fields.
  Let
  \begin{align*}
    dY^{t,x}_s &= V(Y^{t,x}_s) d\mathbf{\arp}_s \\
    Y^{t,x}_t &= x.
  \end{align*}
  Then
  \begin{align*}
    dY^{t,x}_T &= - DY^{t,x}_T V(x) d\mathbf{\arp}_t \\
    dDY^{t,x}_T &= - D^2 Y^{t,x}_T \langle V(x) d\mathbf{\arp}_t, \cdot \rangle
        - D Y^{t,x}_T DV(x) d\mathbf{\arp}_t.
  \end{align*}
\end{lemma}
\begin{proof}
    Take the time derivative of
    \begin{align*}
      Y^{t, Y^{-1,t,x}_T}_T = x,
    \end{align*}
    for the first identity and consider
    the enlarged equation
    \begin{align*}
      dZ = G(Z) d\mathbf{\arp}, 
    \end{align*}
    with $G(x_1,x_2) = \left( V(x_1), DV(x_1) x_2 \right)$
    for the second identity.
\end{proof}

\begin{lemma}
  \label{lem:productRuleMeasureValued}
  Let $\alpha \in (1/3,1/2]$ and $\mathbf{W}$ a geometric $\alpha$-H\"older rough path.
  Let $\rho$ be a solution to the forward equation \eqref{eq:forwardRPDE}
  in the sense of Definition \ref{def:forwardSolutionRough};
  in particular $(\rho_t(f), \rho_t(\Gamma f))$ is controlled for every $f \in C^3_b$.
  Assume moreover that for every $R>0$
  \begin{align*}
    \sup_{||f||_{C^3_b(\R^\dode)}< R} \{ ||\rho_\cdot(f), \rho_\cdot( \Gamma f )||_{W,\alpha} \} < \infty.
  \end{align*}

  Let $\phi \in C^{0,3}_b( [0,T] \times \R^\dode )$ be given, satisfiying for $s \le t$
  \begin{align*}
    \phi(t,x) = \phi(s,x) + \int_s^t \alpha_r(x) dr + \int_s^t (\eta_r(x), \eta'_r(x)) d\mathbf{W}_r,
  \end{align*}
  where $(\eta_t(x), \eta'_t(x)) \in \D^{2\alpha}_W$.
  Assume moreover
  \begin{align*}
    \sup_x \left[ ||\eta(x),\eta'(x)||_{W,\alpha} \right] &< \infty, \\
    \sup_x \left[ ||\Gamma_i \phi(x),\Gamma_i \eta(x)||_{W,\alpha} \right] &< \infty, i=1,\dots,\drp.
  \end{align*}
  and $\eta \in C^{0,3}_b( [0,T] \times \R^\dode )$.

  Then $(M,M'), (N,N') \in \D^{2\alpha}_W$ where
  \begin{align*}
    (M_t)_{i=1,\dots,\drp} &:= \rho_t( \Gamma_i \phi_t ) \\
    (M'_t)_{i,j=1,\dots,\drp} &:= \rho_t( \Gamma_j \Gamma_i \phi_t ) + \rho_t( \Gamma_i (\eta_t)_j ) \\
    (N_t)_{i=1,\dots,\drp} &:= \rho_t( (\eta_t)_i ) \\
    (N'_t)_{i,j=1,\dots,\drp} &:= \rho_t( (\eta'_t)_{ij} ) + \rho_t( \Gamma_j (\eta_t)_i ),
  \end{align*}
  and
  \begin{align*}
    \rho_t( \phi_t )
    =  
    \rho_s( \phi_s )
    +
    \int_s^t (M,M')_r d\mathbf{W}_r
    +
    \int_s^t (N,N')_r d\mathbf{W}_r
    +
    \int_s^t \rho_r( \alpha_r + L\phi_r ) dr.
  \end{align*}
\end{lemma}
\begin{remark}
  Note that with $(\eta,\eta') \equiv 0$, $\alpha \equiv 0$, $\phi(0,x) = f(x)$, this reduces to \eqref{eq:roughForwardEquation}.
\end{remark}
\begin{proof}
  First
  \begin{align*}
    (M_t)_i - (M_s)_i
    &=
    \rho_t( \Gamma_i \phi_t )
    -
    \rho_s( \Gamma_i \phi_s ) \\
    &=
    \rho_s( \Gamma_i \phi_{s,t} )
    -
    \rho_{s,t}( \Gamma_i \phi_s )
    +
    \rho_{s,t}( \Gamma_i \phi_{s,t} ) \\
    &=
    \rho_s( \Gamma_i (\eta_s)_j ) W^j_{s,t}
    -
    \rho_s( \Gamma_j \Gamma_i \phi_s ) W^j_{s,t}
    +
    O(|t-s|^{2\alpha}).
  \end{align*}
  Here we used that
  by assumption $\Gamma_i \phi_{s,t}(x) = \Gamma_i \eta_j(x) W^j_{s,t} + O(|t-s|^{2\alpha})$, uniformly in $x$
  and that $\rho_{s,t}(\Gamma_i f) = \rho_s( \Gamma_j \Gamma_i f ) W^j_{s,t} + O(|t-s|^{2\alpha})$ uniformly over bounded sets of $f$ in $C^3_b$.
  It follows that $(M,M') \in \D^{2\alpha}_W$.
  And analogously for $(N,N')$:
  \begin{align*}
    (N_t)_i - (N_s)_i
    &=
    \rho_t( (\eta_t)_i )
    -
    \rho_s( (\eta_s)_i ) \\
    &=
    \rho_s( (\eta_{s,t})_i )
    -
    \rho_{s,t}( (\eta_s)_i )
    +
    \rho_{s,t}( (\eta_{s,t})_i ) \\
    &=
    \rho_s( (\eta'_s)_{i,j} W^j_{s,t}
    +
    \rho_s( \Gamma_j (\eta_s)_i ) W^j_{s,t}
    +
    O(|t-s|^{2\alpha}),
  \end{align*}
  since $\eta_t \in C^3_b(\R^\dode)$ uniformly in $t$.

  Now for the integral equality, for simplicity take $L = 0$, $\alpha=0$.
  Then
  \begin{align*}
    &\rho_t(\phi_t)
    -
    \rho_s(\phi_s) \\
    &=
    \rho_s(\phi_{s,t})
    +
    \rho_{s,t}(\phi_s)
    +
    \rho_{s,t}(\phi_{s,t}) \\
    &=
    \rho_s( (\eta_s)_i ) W^i_{s,t} + \rho_t( (\eta'_s)_{ij} ) \mathbb{W}^{ij}_{s,t}
    +
    \rho_s(\Gamma_i \phi_s) W^i_{s,t} + \rho_s( \Gamma_j \Gamma_i \phi_s ) \mathbb{W}^{ij}_{s,t}
    +
    \rho_s( \Gamma_i (\eta_s)_j ) W^i_{s,t} W^j_{s,t} 
    \\ 
    & \qquad\qquad\qquad +
    O(|t-s|^{3\alpha}) \\
    &=
    \rho_s( (\eta_s)_i ) W^i_{s,t}
    +
    \left[ \rho_t( (\eta'_s)_{ij} ) + \rho_s( \Gamma_j (\eta_s)_i ) \right] \mathbb{W}^{ij}_{s,t}
    +
    \rho_s(\Gamma_i \phi_s) W^i_{s,t} +
    \left[ \rho_s( \Gamma_j \Gamma_i \phi_s ) + \rho_s( \Gamma_i (\eta_s)_j ) \right] \mathbb{W}^{ij}_{s,t} \\ 
    & \qquad\qquad\qquad +
    O(|t-s|^{3\alpha}).
  \end{align*}
  %The claimed equation then follows from
  Now for every partition $\mathcal{P}$ of $[s,t]$ we have the trivial identity
  \begin{align*}
    \rho_t(\phi_t) -  
    \rho_s(\phi_s)
    =
    \sum_{[u,v] \in \mathcal{P}} \left[ \rho_v(\phi_v) -  \rho_u(\phi_u) \right].
  \end{align*}
  The claimed equality then follows from taking the limit along partitions with mesh-size going to zero.
\end{proof}

\begin{lemma}
  \label{lem:productRuleWeak}
  Let $\mathbf{W}$ be a geometric $\alpha$-H\"older rough path, $\alpha \in (1/3,1/2]$.
  Let $u$ be a weak solution to the backward RPDE \eqref{eq:backwardRPDE}
  in the sense of Definition \ref{def:weakSolutionRough};
  in particular $( \langle u_\cdot, \Gamma^* \phi \rangle, \langle u_\cdot, \Gamma^*\Gamma^* \phi\rangle))$ is controlled for every
  $\phi \in C^3_{\exp}(\R^\dode)$.
  Assume moreover that for every $R>0$
  \begin{align}
    \label{eq:productRuleWeakUniformityOfSolution}
    \sup_{||f||_{C^3_{\exp}(\R^\dode)} < R} \{ ||(\langle u_\cdot, f \rangle, \langle u_\cdot, \Gamma f \rangle)||_{W,\alpha} \} < \infty.
  \end{align}

  Let $\phi \in C^{0,4}_{\exp}([0,T]\times\R^\dode)$ be given
  that satisfies for $s \le t$
  \begin{align*}
    \phi(t,x) = \phi(s,x) + \int_s^t \alpha_r(x) dr + \int_s^t (\eta_r(x), \eta'_r(x)) d\mathbf{W}_r,
  \end{align*}
  where $(\eta_t(x), \eta'_t(x)) \in \D^{2\alpha}_W$.
  Assume moreover
  for some $\delta > 0$
  \begin{align*}
    \begin{split}
      ||\eta(x),\eta'(x)||_{W,\alpha} \lesssim e^{-\delta |x|}, \\
      ||\Gamma^*_i \phi(x),\Gamma^*_i \eta(x)||_{W,\alpha} \lesssim e^{-\delta |x|}, i=1,\dots,\drp.
    \end{split}
  \end{align*}
  In addition assume that $\eta \in C^{0,3}_{\exp}( [0,T] \times \R^\dode )$.

  Then $(M,M'), (N,N') \in \D^{2\alpha}_W$ where
  \begin{align*}
    (M_t)_{i=1,\dots,\drp} &:= \langle u_t, \Gamma^*_i \phi_t \rangle \\
    (M'_t)_{i,j=1,\dots,\drp} &:= \langle u_t, \Gamma^*_j \Gamma^*_i \phi_t \rangle + \langle u_t, \Gamma^*_i (\eta_t)_j \rangle \\
    (N_t)_{i=1,\dots,\drp} &:= \langle u_t, (\eta_t)_i \rangle \\
    (N'_t)_{i,j=1,\dots,\drp} &:= \langle u_t, (\eta'_t)_{ij} \rangle + \langle u_t, \Gamma^*_j (\eta_t)_i \rangle,
  \end{align*}
  and
  \begin{align*}
    \langle u_t, \phi_t \rangle
    =  
    \langle u_s, \phi_s \rangle
    -
    \int_s^t (M,M')_r d\mathbf{W}_r
    +
    \int_s^t (N,N')_r d\mathbf{W}_r
    +
    \int_s^t \langle u_r, \alpha_r - L^* \phi_r \rangle dr.
  \end{align*}
\end{lemma}
\begin{proof}
  By assumption
  \begin{align*}
    |\Gamma^*_i \phi_t(x) - \Gamma^*_i \phi_s(x) - \Gamma^*_i \eta_s(x) W_{s,t}|
    \lesssim
    e^{-\delta|x|} |t-s|^{2\alpha}
  \end{align*}
  Hence
  \begin{align*}
    | \langle u_s, \Gamma^*_i \phi_{s,t} \rangle
    -
    \langle u_s, \Gamma^*_i \eta_j \rangle W^j_{s,t}
    |
    \le
    ||u||_\infty
    ||\beta||_{C^2_b(\R^\dode)}
    ||D\phi_t - D\phi_s - D\eta_s W_{s,t}||_{L^1(\R^\dode)}
    \lesssim |t-s|^{2\alpha}.
  \end{align*}
  Moreover
  \begin{align*}
    | \langle u_{s,t}, \Gamma^*_i \phi_s \rangle 
      -
      \langle u_s, \Gamma^*_j \Gamma^*_i \phi_s \rangle W^j_{s,t}
    |
    \lesssim
    |t-s|^{2\alpha},
  \end{align*}
  since $\phi \in C^{0,4}_{\exp}([0,T] \times \R^\dode)$ (and hence $\Gamma^*_j u_t \in C^{3}_{\exp}(\R^\dode)$ uniformly in $t$) and since $u$ satisfies the uniform bound
  \eqref{eq:productRuleWeakUniformityOfSolution}.
  Then also
  $\langle u_{s,t}, \Gamma^*_i \phi_{s,t} \rangle \lesssim |t-s|^{2\alpha}$ and hence
  \begin{align*}
    (M_t)_i - (M_s)_i
    &=
    \langle u_t, \Gamma^*_j \Gamma^*_i \phi_t \rangle W^j_{s,t}
    +
    \langle u_t, \Gamma^*_i (\eta_t)_j \rangle W^j_{s,t}
    +
    O(|t-s|^{2\alpha}),
  \end{align*}
  hence $(M,M')$ is controlled.
  The argument for $(N,N')$ is similar and
  the proof now finishes as the preceding one.
\end{proof}

\subsection{Rough SDEs}

\begin{lemma}
  \label{lem:jointLift}
  For $\mathbf{W}$ a geometric $\alpha$-H\"older rough path, $\alpha \in (1/3,1/2]$ and a Brownian motion $B$, define $\mathbf{Z} = (Z, \mathbb{Z})$ as
  \begin{equation*}
    Z_{t}=\left( 
    \begin{array}{c}
    B_{t} \\ 
    W_{t}%
    \end{array}%
    \right) ,\,\,\,\mathbb{Z}_{s,t}=\left( 
    \begin{array}{ll}
    \mathbb{B}_{s,t}^{Ito} & \int W\otimes dB \\ 
    \int_{s}^{t}B_{s,t}\otimes dW & \mathbb{W}%
    \end{array}%
    \right)
  \end{equation*}%
  Then
  \begin{enumerate}[(i)]
    \item $\mathbf{Z}$ is well-defined and, almost surely, an $\alpha$-H\"older rough path

    \item
    \begin{equation*}
    \left\vert \rho_{\alpha}
    \left(
    \mathbf{Z}(\mathbf{W}), \mathbf{Z}(\mathbf{\tilde{W}})
    \right)
    \right\vert
    _{L^{q}}\lesssim \rho _{\alpha }\left( \mathbf{W,\tilde{W}}\right) .
    \end{equation*}%

    \item $N_{1;[0,T]}( ||\mathbf{Z}||_p^p )$ has Gaussian tails, uniformly over $\mathbf{W}$ bounded,
    for all $p = \frac{1}{\alpha}$.
  \end{enumerate}
\end{lemma}

\begin{proof}
  This is proven in \cite{DOR},
  the only difference being that there, $\mathbf{Z}$ is only shown to be an $\alpha'$-H\"older rough path,
  for $\alpha' < \alpha$.
  This stems from the fact, that there, a Kolmogorov-type argument
  is applied to the \emph{whole} rough path $\mathbf{Z}$, which in particular contains
  the deterministic path $W$, which explains the decay in perceived regularity.
  
  Being more careful, and
  applying a Kolmogorov-type argument (e.g. Theorem 3.1 in \cite{FH14}) only to the second level, one
  sees that it is actually $\beta$-H\"older continuous, for $\beta < \alpha + 1/2$.
  The first level is trivially $\alpha$-H\"older continuous.
  The claimed continuity in $\mathbf{W}$ is then improved similarly.
\end{proof}

\begin{lemma}[Rough SDE]
  \label{lem:roughSDE}
  Le $\mathbf{W}$ be a geometric $\alpha$-H\"older rough path, $\alpha \in (1/3,1/2]$ and let $\mathbf{Z} = (Z, \mathbb{Z})$ 
  be the joint lift of $\mathbf{W}$ and a Brownian motion $B$, given in the previous Lemma \ref{lem:jointLift}.
  Assume $\sigma_i, \beta_j \in \Lip^3(\R^\dode), i=1,\dots,\dbm, j = 1,\dots, \drp$,
  $b \in \Lip^{1}(\R^\dode)$.
  Let $X = X(\omega)$ be the solution to the rough differential equation
  \begin{align*}
    dX = b(X) dt + V(X) d\mathbf{Z}, 
  \end{align*}
  where $V = (\sigma, \beta)$.
  Then $X$ formally solves the \emph{rough SDE}
  \begin{align*}
    dX = b(X) dt + \sigma(X) dB + \beta(X) d\mathbf{W}.
  \end{align*}

  We have the following properties:
  \begin{itemize}
    \item For all $p \ge 1$, the mapping
      \begin{align*}
        \Cr^\alpha &\to \mathcal{S}^p \\
        \mathbf{W} &\mapsto X,
      \end{align*}
      is locally uniformly continuous.
      Here $||X||_{\mathcal{S}^p}^p := \E[ \sup_{t \le T} |X_t|^p ]$.
      Moreover for every $R>0$ there is $\delta > 0$ such that
      \begin{align*}
        \sup_{||\mathbf{W}||_\alpha < R} \E[ \exp( \delta |X^{\mathbf{W}}|_\infty^2 ) ] < \infty.
      \end{align*}

      If in addition, for $n\ge0$,
      $\sigma_i, \beta_j \in \Lip^{\eta + n}, b \in \Lip^{1+\varepsilon+n}$,
      then the same holds true for $D^n X$.
      %\TODO{ quick check: why is the last statement true? }
    %
    \item For $\mathbf{W}$ the canonical lift of a smooth path $W$, $X$ coincides with the classical SDE solution to
      \begin{align}
        \label{preroughSDE}
        dX_t = b(X_t) dt + \sigma(X_t) dB + \beta(X_t) \dot W_t dt.
      \end{align}
    \item
      Let $c \in \Lip^1(\R^\dode)$, $\gamma \in \Lip^{\eta}(\R^\dode)$ and $g \in \Lip^1(\R^\dode)$, then
      $\int \gamma(X_s) d\mathbf{W}_s$ is a well-defined rough integral, and moreover, for all $p \ge 1$
      \begin{align*}
        [0,T] \times \R^d \times \Cr^\alpha &\to L^p \\
        (t,x,\mathbf{W}) &\mapsto g( X^{t,x}_T ) \exp\left( \int_t^T c(X^{t,x}_r) dr + \int_t^T \gamma(X^{t,x}_r) d\mathbf{W}_r \right),
      \end{align*}
      is continuous, uniformly in $t,x$ and locally uniformly in $\mathbf{W}$.
  \end{itemize}

\end{lemma}
\begin{proof}
  If $b \in C^{1+\varepsilon}_b(\R^\dode)$, some $\varepsilon > 0$
  this is shown in \cite[Theorem 10]{DOR}.
  Now, for $b \in C^1_b(\R^\dode)$ the same proof works,
  one just needs to use the improved result
  on RDEs with drift in \cite[Proposition 3]{FO09}.
\end{proof}

\begin{lemma}
  \label{lem:expDecaySurvivesFeynmanKac}
  Let $\mathbf{W}$ be a geometric $\alpha$-H\"older rough path, $\alpha \in (1/3,1/2]$.
  Let $X^{t,x}$ be the solution to the rough SDE (Lemma \ref{lem:roughSDE})
  \begin{align*}
    dX_t = b(X_t) dt + \sigma(X_t) dB + \beta(X_t) d\mathbf{W}, \qquad
    X^{t,x}_t = x.
  \end{align*}
  Let $n\ge 0$ and
  assume 
  $c \in \Lip^{n}(\R^\dode), \gamma_k \in \Lip^{ 2 + n }(\R^\dode)$,
  $\sigma_i, \beta_j \in \Lip^{3 + [(n-1) \vee 0 ]}(\R^\dode)$,
  $b \in \Lip^{1 + [(n-1) \vee 0 ]}$.
  Then for every $\phi \in C^n_{\exp}(\R^\dode)$ the function
  \begin{align*}
    \psi(x) := \E[ \phi(X^{t,x}_T) \exp\left( \int_t^T c\left( X^{t,x}_r \right) dr +\int_t^T \gamma \left( X^{t,x}_r \right) d\mathbf{W}_r \right) ],
  \end{align*}
  is again in $C^n_{\exp}(\R^\dode)$, with $||\psi||_{C^n_{\exp}(\R^\dode)}$
  bounded uniformly for $t\le T$ and $||\mathbf{W}||_\alpha$ bounded.
\end{lemma}
\begin{proof}
  For $n=0$, let $C_1 > 0$ such that $|\psi(x)| \le C_1 \exp( - \frac{1}{C_1} |x| )$.
  Then
  \begin{align*}
    |\psi(x)|
    &= \E[ \phi(X^{t,x}_T) \exp\left( \int_t^T c\left( X^{t,x}_r \right) dr +\int_t^T \gamma \left( X^{t,x}_r \right) d\mathbf{W}_r \right) ] \\
    &\le C_1 \E[ \exp( -\frac{1}{C_1} |X^{t,x}_T| ) \exp( \dots ) ] \\
    &\le C_1 \exp( -\frac{1}{C_1} x ) \E[ \exp( \frac{1}{C_1} |X^{t,x}_T - x| ) \exp( \dots ) ] \\
    &\le C_1 \exp( -\frac{1}{C_1} x ) \E[ \exp( C_2 \left( 1 + N_{1,[t,T]}(\mathbf{Z}) \right) + T ||c||_\infty ) ] \\
    &\le C_1 \exp( -\frac{1}{C_1} x ) \E[ \exp( C_3 N_{1,[0,T]}(\mathbf{Z}) ) ],
  \end{align*}
  where we used \eqref{eq:infinityBoundBoundedVF} for the 4th line
  This concludes the argument, since the expectation is finite by Lemma \ref{lem:jointLift}, uniformly for $||\mathbf{W}||_\alpha$ bounded.
  The case $n\ge1$ follows similarly, by differentiating under the expectation.
\end{proof}

\begin{lemma}
  \label{lem:expDecaySurvivesFeynmanKacII}
  Let $\mathbf{W}$ be a geometric $\alpha$-H\"older rough path, $\alpha \in (1/3,1/2]$.
  Assume $\sigma_i, \beta_j \in \Lip^{3}(\R^\dode), i=1,\dots,\dbm,j=1,\dots,\drp$,
  $b \in \Lip^1$.
  Let $X^{t,x}$ be the solution to the rough SDE (Lemma \ref{lem:roughSDE})
  \begin{align*}
    dX = \sigma(X) dB + \beta(X) \mathbf{W}.
  \end{align*}
  Let $n\ge 0$.
  For every $\varphi \in C^n_{\exp}(\R^\dode)$, any $q\ge 1$ the function
  \begin{align*}
    \psi(x) := \E[ ||\varphi||_{C^n_b( M(x) ) }^q ],
  \end{align*}
  is in $C^0_{\exp}$.
  Here $M(y) := \{ x: \inf_{r \in [t,T]} |X^{t,x}_r| - 1 \le |x| \le \sup_{r \in [t,T]} |X^{t,x}_r| + 1 \}$.
\end{lemma}
\begin{proof}
  This follows from
  \begin{align*}
    ||\varphi||_{ L^\infty( M(x) ) }
      &\le C \exp( - \delta \inf_{r \in [t,T]} |X^{t,x}_r| ) \\
      &\le C \exp( - \delta \inf_{r \in [t,T]} |X^{t,x}_r| ) \\
      &\le C \exp( -\delta |x| ) \exp( \delta \sup_{r \in [t,T]} |X^{t,x}_r - x| ).
  \end{align*}
\end{proof}

\bigskip \textbf{Acknowledgement:}

All authors acknowledge support from DFG Priority Program 1324.
P. Friz received funding from the European Research Council under the European 
Union's Seventh Framework Programme (FP7/2007-2013) / ERC grant agreement nr. 258237.

\end{document}